\documentclass[12pt, twoside, leqno]{article}

\usepackage{amsmath,amsthm}
\usepackage{amssymb}

\usepackage{enumerate}

\usepackage{graphicx}

\providecommand{\diam}{\mathop{\rm diam}\nolimits}
\providecommand{\tw}[1]{\mathop{\langle#1\rangle_3}}
\providecommand{\grid}{\mathbb{N}_0^3}
\providecommand{\cont}[1]{\mathop{\stackrel{#1}{\frown}}}
\providecommand{\notcont}[1]{\mkern-1mu\not\mathrel{\mkern-2mu\cont{#1}}\mkern1mu}

\newtheorem{theorem}{Theorem}[section]
\newtheorem{corollary}[theorem]{Corollary}
\newtheorem{lemma}[theorem]{Lemma}
\newtheorem{proposition}[theorem]{Proposition}
\newtheorem{question}[theorem]{Question}

\newtheorem{conjecture}[theorem]{Conjecture}

\theoremstyle{definition}

\numberwithin{equation}{section}

\begin{document}


\baselineskip=17pt


\title{Commuting Contractive Families}

\author{Luka Mili\'{c}evi\'{c}\\
Trinity College\\ 
Cambridge CB2 1TQ\\
United Kingdom\\
E-mail: lm497@cam.ac.uk}

\date{}

\maketitle


\renewcommand{\thefootnote}{}



\renewcommand{\thefootnote}{\arabic{footnote}}
\setcounter{footnote}{0}


\begin{abstract}
A family $f_1,...,f_n$ of operators on a complete metric space $X$ is called contractive if there exists a positive $\lambda < 1$ such that for any $x,y$ in $X$ we have $d(f_i(x),f_i(y)) \leq \lambda d(x,y)$ for some $i$. Austin conjectured that any commuting contractive family of operators has a common fixed point, and he proved this for the case of two operators.\\
Our aim in this paper is to show that Austin's conjecture is true for three operators, provided that $\lambda$ is sufficiently small.\\
\end{abstract}

\section{Introduction}
Let $(X,d)$ be a (non-empty) complete metric space. Given $n$ functions $f_1, f_2, \dots, f_n : X \rightarrow X$, and a real number $\lambda \in (0,1)$, we call $\{f_1, f_2, \dots, f_n\}$ a \emph{$\lambda$-contractive} family if for every pair of points $x, y$ in $X$ there is $i$ such that $d(f_i(x), f_i(y)) \leq \lambda d(x,y)$. We say that $\{f_1, f_2,\dots, f_n\}$ is a \emph{contractive} family if it is $\lambda-$contractive for some $\lambda \in(0,1)$. Furthermore, when $f$ is a function on $X$ and $\{f\}$ is a contractive family we say that $f$ is a contraction. Recall the well-known theorem of Banach~\cite{Banach} stating that any contraction on a complete metric space has a fixed point. Finally, by an \emph{operator} $f$ on $X$ we mean a continuous function $f:X \to X$.\\
Stein~\cite{Stein} considered various possible generalizations of this result. In particular, he conjectured that for any contractive family of operators $\{f_1,f_2,\dots, f_n\}$ on a complete metric space there is a composition of the functions $f_i$ (i.e. some word in $f_1,\dots,f_n$) with a fixed point. We refer to this statement as Stein's conjecture. However, in ~\cite{Austin}, Austin shown that this is in fact false. Recently, the author of this paper showed~\cite{Milicevic} that the Stein's conjecture fails even for compact metric spaces. But Austin also showed that if $n = 2$ and $f_1$ and $f_2$ commute the conjecture of Stein will hold.\\
With this in mind, we say that $\{f_1, f_2, \dots, f_n\}$ is \emph{commuting} if every $f_i$ and $f_j$ commute. 
\begin{theorem}[\cite{Austin}] Suppose that $\{f,g\}$ is a commuting contractive family of operators on a complete metric space. Then $f$ and $g$ have a common fixed point.\end{theorem}
Let us mention another result in this direction, which was proved by Arvanitakis in~\cite{Arva} and by Merryfield and Stein in~\cite{Merry2}.
\begin{theorem}\label{genban}[\cite{Arva}, \cite{Merry2}, Generalized Banach Contraction Theorem] Let $f$ be a function from a complete metric space to itself, such that $\{f, f^2, \dots, f^n\}$ is a contractive family. Then $f$ has a fixed point.\end{theorem}
Note that there is no assumption of continuity in the statement of the Theorem~\ref{genban}. We also remark that Merryfield, Rothschild and Stein proved this theorem for the case of operators in~\cite{Merry1}. Furthermore, Austin raised a question which is a version of Stein's conjecture, and generalizes these two theorems in the context of operators.
\begin{conjecture}[\cite{Austin}] Suppose that $\{f_1, f_2, \dots, f_n\}$ is a commuting contractive family of operators on a complete metric space. Then $f_1, f_2, \dots, f_n$ have a common fixed point.\end{conjecture}

Let us now state the result that we establish here, which proves the case $n=3$ and $\lambda$ sufficiently small: 

\begin{theorem}\label{thmMain}Let $(X,d)$ be a complete metric space and let $\{f_1, f_2, f_3\}$ be a commuting $\lambda-$contractive family of operators on X, for a given $\lambda \in (0,10^{-23})$. Then $f_1, f_2, f_3$ have a common fixed point.\end{theorem}

We remark that such a fixed point is necessarily unique. 

\section{Main goal, notation and definitions}
In this section we will give a statement that implies the theorem we want to prove and we will establish the most needed notation and definitions that will be used extensively throughout the proof of the Theorem~\ref{thmMain}. We write $\mathbb{N}_0$ for the set of nonnegative integers and for a positive integer $N$, $[N]$ stands for the set $\{1,2, \dots, N\}$.\\
When $a$ is an ordered triple of nonnegative integers and $x \in X$, we define $a(x) = f_1 ^ {a_1} \circ f_2 ^ {a_2} \circ f_3 ^ {a_3} (x)$. Since our functions commute, we have $a(b(x)) = (a+b)(x)$, where $b \in \mathbb{N}_0 ^ 3$, as well.\\
Pick arbitrary point of our space $p_0 \in X$, and define new pseudometric space (abusing the notation slightly) $G(p_0) = (\mathbb{N}_0 ^ 3, d)$, where $d(a, b) = d(a(p_0),b(p_0))$, when $a, b$, are ordered triples of non-negative integers. Therefore, we will actually work on an integer grid instead. Define $e_i$ to be triple with 1 at position $i$, and zeros elsewhere. Now, we will prove a few basic claims which will tell us what in fact our main goal is.
\begin{proposition}\label{conj1}Let $(X,d)$ be a complete metric space and $\lambda \in (0, 10^{-23})$ given, with $f_1, f_2, f_3: X \rightarrow X$ which form a commuting $\lambda$-contractive family. Then for some $i$, $f_i$ has a fixed point. \end{proposition}
\begin{proposition} If Proposition~\ref{conj1} holds, so does Theorem~\ref{thmMain}.\end{proposition}
\begin{proof} Without loss of generality, we have fixed point $x$ of $f_1$. Thus, define $X_1$ to be the set of all fixed points of $f_1$. It is closed subspace of $X$, hence complete. Further, $s \in S_1$ implies $f_1(f_i(s)) = f_i(f_1(s)) = f_i(s)$, so $f_i(s) \in S_1$, hence the other two functions preserve $S_1$, and form a $\lambda$-contractive family themselves, so $f_2$ has fixed point in $S_1$, and repeat the same argument once more to obtain a common fixed point.\end{proof}
\begin{proposition}\label{keyprop}Let $(\mathbb{N}_0 ^ 3, d)$ be a pseudometric space and constant $\lambda \in (0, 10^{-23})$, such that given any $a,b \in \mathbb{N}_0 ^ 3$, there is $i \in [3]$ for which $\lambda d(a, b) \geq d(a + e_i, b + e_i)$. Then there is a Cauchy sequence $(x_n)_{n \geq 1}$ in this space, such that $x_{n+1} - x_{n}$ is always an element of $\{e_1, e_2, e_3\}$.\end{proposition}
\begin{proposition} If the Proposition~\ref{keyprop} holds, so does Theorem~\ref{thmMain}.\end{proposition}
\begin{proof} It suffices to show that the Proposition~\ref{keyprop} implies the Proposition~\ref{conj1}. Let $(X,d)$ be a metric space as in the Proposition~\ref{conj1}, along with three functions acting on it. Pick arbitrary point $p_0 \in X$, and consider pseudometric space $G(p_0)$ defined before. By Proposition~\ref{keyprop}, we have a Cauchy sequence $(x_n)_{n \geq 1}$, with property above. So, $(x_n(p_0))$ is Cauchy in $X$. Without loss of generality, we have change by $e_1$ infinitely often, say for $(x_{n_i})_{i \geq 1}$, $x_{n_i + 1} =x_{n_i} + e_1$ holds. As $X$ is complete, $x_n(p_0)$ converges to some $x$. Hence $x_{n_i}(p_0)$ converges to $x$, and so does $f_1(x_{n_i}(p_0))$, but $f_1$ is continuous, thus, $f_1(x) = x$.\end{proof}
Therefore, Proposition~\ref{keyprop} is what is sufficient to prove. The integer grid has a lot of structure itself, and the following definitions aim to capture some of it and to help us establish the claim.\\
Let $x$ be a point in the grid. Define $\rho(x) $ to be the maximal of the distances $d(x, x+e_1), d(x, x+e_2), d(x, x+e_3)$. As we shall see in the following section, $\rho$ will be of fundamental importance. Given $x$ in the grid, we define $N(x) = \{x+e_1, x+e_2, x+e_3\}$ and refer to this set as the \emph{neighborhood} of $x$.\\
Let $S$ be a subset of the grid. Given $k \in [3]$, we say that $S$ is a \emph{$k$-way} set, if for all $s\in S$, precisely $k$ elements of $N(s)$ are in $S$. We denote the unique three way set starting from $x$ by $\tw{x} = \{x + k: k \in \grid\}$.

\section{Overview of the proof of Proposition~\ref{keyprop}}
The proof of the Proposition~\ref{keyprop} will occupy the most of the remaining part of the paper. To elucidate the proof, we will structure it in a few parts. The short first section will show our strategy in the proof along with some of the basic ideas. The second part will be about $k-$way sets and how they interact with each other. Afterwards, we shall be dealing with local structure, namely we shall show existence or non-existence of certain finite sets of points, and our main means to this end will be $k-$way sets. Finally, after we have clarified the local structure sufficiently well, we will be able to obtain the final contradiction. Let us now be more precise and elaborate on these parts of the proof.\\
First of all, we shall establish a few basic facts about $\rho$, most importantly $\mu = \inf \rho(x) > 0$, where $x$ ranges over all points. The important fact for the proof of this statement is a lemma that says $d(x,y) \leq (\rho(x) + \rho(y))/(1-\lambda)$. This fact will be the pillar of the proof, and on the other hand the lemma mentioned will be used throughout the proof. The basic idea which is introduced in this section is to create sets of points by contracting with some previously chosen ones (by contracting a pair $x,y$ we mean choosing suitable function $f$ in our family such that $d(f(x),f(y)) \leq \lambda d(x,y)$). By doing so, we will be able to construct $k-$ways sets of bounded diameter.\\  
After that, we shall prove a few propositions about the $k-$way sets. For example, if we have $3-$way set of bounded diameter then it contains $2-$way subset of much smaller diameter, in the precise sense discussed afterwards. At the first glance, it seems that we have lost a dimension by doing this, however, we shall also show that if we have $2-$way set of sufficiently small diameter, we can obtain 3-way set of small diameter as well. So, for example, given $K$ and provided $\lambda$ is small enough we cannot have 3-way sets of diameter $K\mu$, and we cannot have $2-$way sets of diameter $\lambda K\mu$ inside every $3-$set. From this point on we shall combine the results and approaches of these two parts in the proof. Most of the claims that we establish afterwards will either show that certain finite configuration (by which we mean finite set of points with suitable distances between) exists or does not exist, and we do so by supposing contrary, contracting new points with the given ones and finding suitable $k$-way sets, which give us a contradiction.\\
As the basic example of this method, we note that each point $x$ induces a 1-way set of diameter at most $2\rho(x) / (1-\lambda)$, more importantly, such a set exists in every 3-way set. With a greater number of suitable points we are able to induce bounded $k-$way sets for larger $k$.
\begin{figure}
\includegraphics[scale = 0.4]{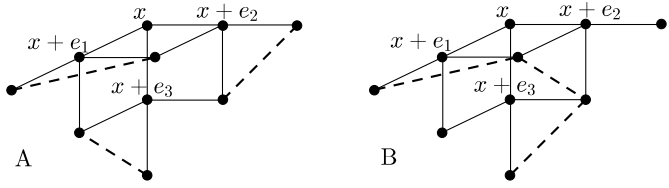}
\caption{Examples of diagrams}
\label{exdiagram}
\end{figure}
Using the facts established, we prove the existence or non-existence of specific finite sets. Gradually, we learn more about the local structure of the grid considered. For example, for some constant $C$ (independent of $\lambda$) we have $y$ with $\rho(y) \leq C\mu, d(y + e_1, y + e_2) \leq \lambda C\mu$, provided $\lambda$ is small enough. Similarly, we shall establish that there is no point $y$ with $\rho(y) \leq C\mu, \diam N(y) \leq \lambda C\mu$, for suitable $\lambda, C$. Such points will be used on a few places in the later part of the proof and in the final argument to reach the contradiction. At this point we introduce the notion of \emph{diagram} of a point $x$, namely it is the information about the contractions in $\{x\}\cup N(x)$. The diagrams will be shown in the figures, and usually the dash lines will imply that corresponding edge is result of a contraction. In the Figure~\ref{exdiagram} we give an example of two diagrams\footnote{In other figures we shall not denote the points on the diagram itself, however, the coordinate axes will always be the same.}, the left one, denoted by A, tells us among others that $x + e_1, x + e_2$ are contracted by 1 (i.e. $d(x + 2e_1, x + e_1 + e_2) \leq \lambda d(x + e_1, x + e_2)$). The claims established so far allow us to have a very restricted number of possibilities for diagrams, and one of the possible strategies will then be to classify the diagrams, see how they fit together and establish the existence of a 1-way Cauchy sequence. The most important claim that we use for rejecting diagrams is the following Proposition.
\begin{proposition}\label{tpsO} Given $K \geq 1$, suppose we have $x_0, x_1, x_2, x_3$ such that $\diam\{x_i + e_j:i,j \in [3], i\not=j\} \leq \lambda K\mu$. Furthermore, suppose $\rho(x_0) \leq K\mu$ and that $d(x_0, x_i) \leq K\mu$ for $i\in[3]$. Let $\{a,b,c\} = [3]$.\\
Provided $\lambda < 1/(820 C_1 K)$, whenever there is a point $x$ which satisfies $d(x + e_a, x + e_b) \leq \lambda K \mu$ and $d(x, x_0)\leq K\mu$, then we have $d(x + e_c, x_c+e_c) \leq 16 \lambda K \mu$.\end{proposition}
The final part of the proof is based on the following claim:
\begin{proposition}\label{casescorO} Fix arbitrary $x_0$ with $\rho(x_0) < 2\mu$. Given $K \geq 1$, when $i\in[3]$, define $S_i(K,x_0) = \{y: d(x_0,y) \leq K\mu, d(y, y + e_i) \leq K\mu\}$. Provided $1 > 10\lambda K C_1$, in every $\tw{z}$ there is $t$ such that $d(t,x_0) \leq 3K\mu$, but for some $i$ we have $s \notcont{i} t$ when $s \in S_i(K,x_0)$. \end{proposition}
Using the point $t$, whose existence is provided, we shall discuss the cases on $d(t + e_1, t + e_2)$ being large or small. Both help us to reject many diagrams and then to establish the contradiction in a straight-forward manner.\\
To sum up, the basic principle here is that contractions ensure that we get specific finite sets. On the other hand, certain finite sets empower contractions further, allowing us to construct $k-$way sets of small diameter. Therefore, if we are to establish a contradiction, we can expect a dichotomy; either we get finite sets that imply global structure that is easy to work with, or we do not have such sets, and we impose strong restrictions on the local structure of the grid.\\

We are now ready to start the proof of the Proposition~\ref{keyprop}. The proof will run for the most of the paper, ending in the Section 8.

\begin{proof}[Proof of the Proposition~\ref{keyprop}]Suppose contrary, there is no 1-way Cauchy sequence in the given pseudometric space on $\grid$. This condition will, as we shall see, imply a lot about the structure of the space, and we will start by getting more familiar with the function $\rho$, which will, as it was already remarked, play a fundamental role.

\section{Basic finite contractive configurations arguments and properties of $\rho$}
In this section we establish few claims about $\rho$, together with some claims which will come in handy at several places throughout the proof.
\begin{lemma}[(Furthest neighbor inequality - FNI)]\label{fni}Given $x,y$ in the grid we have $d(x,y) \leq (\rho(x) + \rho(y))/(1-\lambda)$.\end{lemma}
\begin{proof} Let $i$ be such that $\lambda d(x,y) \geq d(x+e_i, y + e_i)$, which we denote from now on by $x\cont{i}y$, and say that $i$ \emph{contracts} $x,y$\footnote{We also say $x,y$ is \emph{contracted} by $i$, or $x,y$ is \emph{contracted in the direction $e_i$}.}. Using the triangle inequality a few times yields $d(x,y) \leq d(x, x + e_i) + d(x+e_i, y+e_i) + d(y+e_i, y) \leq \lambda d(x,y) + \rho(x) + \rho(y)$, which implies the result.\end{proof}

Similarly to $x \cont{i} y$, we write $x \notcont{i}y$ to mean that $d(x + e_i, y+e_i) > \lambda d(x,y)$.

\begin{lemma}\label{1contrLemma}Let $x, y$ be any two points in the grid. Then we can find a 1-way subset $S$, such that $y\in S$ and given $\epsilon > 0$ we have $d(s,x) \leq \frac{1}{1-\lambda}\rho(x) + \epsilon$ for all but finitely many $s\in S$.\end{lemma}

\begin{proof} Consider the sequence $(x_n)_{n\geq0}$ defined inductively by $x_0 = y$ and for any $k \geq 0$, we set $x_{k+1} = x_k + e_i$ when $i$ contracts $x$ and $x_k$. By induction on $k$ we claim that $d(x, x_k) \leq \lambda^k d(x, y) + \rho(x)/(1-\lambda)$.\\
Case $k = 0$ is clear as $\rho(x_0) \geq 0$. If the claim holds for some $k$ and $x_k \cont{i} x$, then by the triangle inequality we have $d(x_{k+1}, x) \leq d(x_{k+1}, x+e_i) + d(x+e_i, x) \leq \lambda d(x_k, x) + \rho(x_0) \leq \lambda ^{k+1}d(x, y) + \lambda \rho(x_0) / (1-\lambda) + \rho(x_0) \leq \lambda ^ {k+1} d(x, y) + \rho(x_0) / (1-\lambda)$, as desired.\\
Now, take $n$ sufficiently large so that $\rho(x_0)/(1-\lambda) + \lambda ^ n d(x,y) \leq \frac{1}{1-\lambda}\rho(x_0) + \epsilon$. Hence $d(x_k, x) \leq \frac{1}{1-\lambda}\rho(x_0)+\epsilon$ for all $k \geq n$, so choose $(x_k)_{k\geq 0}$ as the desired set. 
\end{proof}

\begin{proposition} Given any $x$ in the grid, we have $\rho(x) > 0$.\end{proposition}

\begin{proof} Suppose contrary, $\rho(x) = 0$ for some $x$. Then the Lemma~\ref{1contrLemma} immediately gives a 1-way Cauchy sequence, which is a contradiction. \end{proof}

\begin{proposition}\label{mupos} The infimum $\inf\{\rho(x): x \in \mathbb{N}_0^3\}$ is positive.\end{proposition}

This result is one the of crucial structural properties for the rest of the proof, and having it in mind, we will try either to find small $\rho$, or use the structure implied to get a Cauchy sequence, which will yield a contradiction. To prove this statement, we use the Lemma~\ref{1contrLemma}, the difference being that we now contract with many different points of small $\rho$ instead of just one.

\begin{proof} Suppose contrary, hence we get $(y_n)_{n\geq1}$ such that $\rho(y_n) < 1/n$. As $\rho$ is always positive, we can assume that all elements of the sequence are distinct.\\
We define a 1-way sequence $(x_k)_{k \geq 0}$ as follows: start from arbitrary $x_0$ and contract with $y_1$ as in the proof of Lemma~\ref{1contrLemma} until we get a point $x_{k_1}$ with $d(x_{k_1}, y_1) \leq 2 \rho(y_1)/(1-\lambda)$ (such a point exists by Lemma~\ref{1contrLemma}). Now, start from $x_{k_1}$ and contract with $y_2$ until we reach $x_{k_2}$ with $d(x_{k_2}, y_2) \leq 2 \rho(y_2) / (1-\lambda)$. We insist that $k_{i+1} > k_i$ for all possible $i$, so that, proceeding in this way, one defines the whole sequence. Recalling the estimates in the proof of the Lemma~\ref{1contrLemma}, we see that for $k_i \leq j \leq k_{i+1}$ we have $d(x_j, y_{i+1}) \leq d(x_{k_i}, y_{i+1}) + \rho(y_{i+1})/ (1-\lambda) \leq d(x_{k_i}, y_i) + d(y_i, y_{i+1}) + \rho(y_{i+1}) / (1-\lambda)$. So by FNI, we see that $d(x_j, y_{i+1}) \leq (3\rho(y_i) + 2\rho(y_{i+1})) / (1-\lambda) \leq \frac{5}{i(1-\lambda)}$. Hence, if we are given any other $x_{j'}$ with $k_{i'} \leq j' \leq k_{i' + 1}$ by the triangle inequality and FNI we see that $d(x_j, x_{j'}) \leq \frac{6}{1-\lambda} (1/i + 1/i')$, which is enough to show that the constructed sequence is 1-way Cauchy.\end{proof}

We will denote $\inf \rho$, where the infimum is taken over the whole grid, by $\mu$. The proposition we have just proved gives $\mu > 0$. 

\section{Properties of and relationship between $k$-way sets}
The following propositions reflect the nature of $k$-way sets. These both confirm their importance for the problem and will prove to be useful at various places throughout the proof. 

\begin{proposition}\label{32prop}If $\tw{\alpha}$ is a 3-way set of diameter $D$, then it contains a 2-way subset of diameter not greater than $\lambda C_1 D$, where $C_1 = 49158$.\end{proposition}

\begin{proof}The proof will be a consequence of the Proposition~\ref{prop3} and Lemma~\ref{lemma4}, each of which needing its own auxilliary lemma. Let us start by establishing:

\begin{proposition}\label{prop3}If the conclusion of the Proposition~\ref{32prop} does not hold, then given $x, y \in \tw{\alpha}$ and distinct $i,j \in [3]$, there is $z \in \tw{\alpha}$ with $d(x, z + e_i) > 2 \lambda D$ and $d(y, z + e_j) > 2 \lambda D$.\end{proposition}

The purpose of this Proposition is to provide us with a finite set of points which will then be used to induce a 2-way set of the wanted diameter, by contractions. To prove this claim, we examine two cases on the distance between $x$ and $y$, one when $d(x,y) > 5\lambda D$ and the other $d(x,y) \leq 5 \lambda D$.


\begin{proof}[Proof of Proposition~\ref{prop3}] As noted above, we look at two cases on $d(x,y)$.


\emph{Case 1: $d(x,y) > 5 \lambda D$.} We shall actually prove something more general in this case -- if $d(x,y) > 5 \lambda D$ and we cannot find a desired point $z$, then we get a 3-way subset $T$ of $\tw{\alpha}$ of diameter not greater than $4\lambda D$.\\
Suppose there is no such $z$, hence for all $z \in \tw{\alpha}$ either $d(x, z+e_1) \leq 2\lambda D$ or $d(y, z+e_2) \leq 2\lambda D$ is true. We can color all points $t$ in this 3-way set by $c(t) = 1$ if $d(t, x) \leq 2\lambda D$, by $c(t) = 2$ if $d(t, y) \leq 2\lambda D$, and $c(t) = 3$ otherwise. This is well-defined as the triangle inequality prevents first two conditions holding simultaneously. Thus, for any $z$ either $c(z + e_1) = 1$ or $c(z + e_2) = 2$. Also given any two points $z, t$ in the grid such that $t \cont{j}z$, and whose neighbors take only colors 1 and 2, it cannot be that $c(z+e_j) \not= c(t + e_j)$, as otherwise, w.l.o.g. $c(t + e_j) = 1, c(z + e_j) = 2$. Then we get $d(x, y) \leq d(x, t + e_j) + d(t + e_j, z + e_j) + d(z + e_j, y) \leq 5 \lambda D$, which is a contradiction. Thus for any such $z$ and $t$, there is an $i$ such that $c(t+e_i) = c(z+e_i)$.\\
The following auxiliary lemma tells us that all such colorings are essentially trivial. (Note that we are still in the Case 1 of the proof of Proposition~\ref{prop3}.)

\begin{lemma}\label{strongcolLemma}Let $c:\tw{\beta}\to[3]$ be a colouring such that
\begin{enumerate}
\item Given $z \in \tw{\beta}$ either $c(z + e_1) = 1$ or $c(z + e_2) = 2$,
\item Given $z,t \in \tw{\beta}$ such that neighbors of $z, t$ take only colors 1 and 2, then $c(z + e_i) = c(t + e_i)$ for some $i$.
\end{enumerate}
Then there is a 3-way subset of $\tw{\beta}$ which is either entirely coloured by 1 or entirely coloured by 2.
\end{lemma}

\begin{proof}[Proof of the Lemma~\ref{strongcolLemma}]
We denote the coordinates by superscripts. Given non-negative integers $a \geq \beta^{(3)},b \geq \beta^{(1)} + \beta^{(2)}$ denote $\mathcal{L}(a,b) = \{z \in\grid: z^{(3)} = a, z^{(1)} + z^{(2)} = b\}$. Hence such a line must be colored as $c(b - \beta_2, \beta_2, a) = 1, c(b - \beta_2 - 1, \beta_2 +1, a) = 1, \dots c(t + e_1 - e_2) = 1, c(t)$ arbitrary$, c(t + e_2 - e_1) = 2,\dots ,c(\beta_1, b - \beta_1, a) = 2$, for some point $t$. If all $z$ in $\mathcal{L}(a,b)$, with $z^{(1)} \geq \beta^{(1)} + 3, z^{(2)} \geq \beta^{(2)} + 3$ are colored by 1, say that $\mathcal{L}(a,b)$ is \emph{1-line}. Similarly, if these are coloured by 2, call it a \emph{2-line}, and otherwise \emph{1,2-line}.\\
Observe that if $ \mathcal{L}(a,b)$ is 1,2-line for $a \geq \beta^{(3)}$ and $b > \beta^{(1)} + \beta^{(2)} + 10$, then $\mathcal{L}(a+1,b-1)$ is not 1,2-line, for otherwise we have:
\begin{enumerate}
\item $(\beta^{(1)}, b - \beta^{(1)}, a), (\beta^{(1)} + 1, b - \beta^{(1)} - 1, a), (\beta^{(1)}, b - \beta^{(1)} - 1, a + 1)$ are colored by 1,
\item $(b- \beta^{(2)}, \beta^{(2)}, a), (b- \beta^{(2)}-1, \beta^{(2)}+1, a), (b-\beta^{(2)}-1, \beta^{(2)}, a + 1)$ are colored by 2,
\end{enumerate}
which is impossible by the second property of the colouring.\\
Suppose we have a 1,2-line $\mathcal{L}(a,b)$ for $a > \beta^{(3)}, b > \beta^{(1)} + \beta^{(2)} + 20$. Then $\mathcal{L}(a+1,b-1)$ and $\mathcal{L}(a-1,b+1)$ are either 1- or 2-lines. But as above we can exhibit $x', y'$ such that $x' + e_1, x'+e_2, y' + e_3$ are of color 1 while $y' + e_1, y'+e_2, x' + e_3$ are of 2, or we can find $x',y'$ for which $x' + e_1, x'+e_2, x' + e_3$ have $c = 1$, and $y' + e_1, y'+e_2, y' + e_3$ are coloured by 2. So, there can be no such 1,2-lines. Further, by the same arguments we see that $\mathcal{L}(a, s-a)$ for fixed $s$ must all be 1-lines or all 2-lines, for $a > \beta_3 + 1$, and that in fact only one of these possibilities can occur, hence we are done.\end{proof}
Applying the Lemma~\ref{strongcolLemma} immediately solves the case 1 of the proof.\\


\emph{Case 2: $d(x,y) \leq 5\lambda D$.} Suppose contrary. Then, in particular for any $z$, we have $d(x, z+e_1) \leq 7 \lambda D$ or $d(x, z+e_2) \leq 7 \lambda D$. Further, we must have $z$ such that $d(x, z+e_{i_1}) , d(x, z+e_{i_2}) > 10 \lambda D$ holds for some distinct $i_1, i_2 \in [3]$. Take such a $z$, and without loss of generality $i_1 = 2, i_2 = 3$. So $d(z + e_1, x) \leq 7\lambda D$. Hence $d(z + (-1, 1, 1), x) \leq 7 \lambda D$ and contracting $z, z + (-1, 0, 1)$ gives $d(z + (-1, 0 ,2), x) > 9\lambda D$. Now contract $z, z + (-1,1,0)$ to get $d(z, z + (-1, 2, 0)) > 9 \lambda D$. However, this is a contradiction, as both $z + (-1, 1, 0) + e_1$ and $z + (-1, 1, 0) + e_2$ are too far from $x$.\\

Having settled both cases on the $d(x,y)$, the Proposition is proved.\end{proof}


If there is $x \in\tw{\alpha}$ such that for some $x' \in\tw{\alpha}$ and for all points $y \in \tw{x'}$ we have $d(x, y) \leq 5\lambda D$, we are done. Hence, we can assume that for all $x, x'\in\tw{\alpha}$ there is $y\in\tw{x'}$ which violates the above distance condition.\\
Take now arbitrary $x_0 \in \tw{\alpha}$. Due to observation we have just made, we know that for any $i \in [3]$ there is an $x_i \not = x_0$ such that $d(x_i + e_i, x_0 + e_i) > 5 \lambda D$ and to be on the safe side, assume that the neighborhoods of $x_0, x_1, x_2, x_3$ are all disjoint. Now, by the Proposition~\ref{prop3}, given $i\not=j$ in $[3]$, we can find $x_{i,j} \in \tw{\alpha}$ such that $d(x_{i,j} + e_i, x_0 + e_i) > 2 \lambda D, d(x_{i,j} + e_j, x_i + e_j) > 2 \lambda D$. Now, let $y$ be any element of the 3-way set generated by $\alpha$. Take $i$ which contracts $x_0, y$, implying $d(x_0 + e_i, y + e_i) \leq \lambda D$. Hence, by triangle inequality $d(x_i + e_i, y + e_i) > \lambda D$, so $x_i, y$ must be contracted by some $j \not = i$. Using the triangle inequality once more, we get $d(x_{i,j} + e_j, y + e_j) > \lambda D$ and by construction $d(x_{i,j} + e_i, y + e_i) \geq d(x_{i,j} + e_i, x_0 + e_i) - d(x_0 + e_i, y + e_i)> \lambda D$, therefore for $k \not = i,j$, $d(y+e_k, x_{i,j} + e_k) \leq \lambda D$. We are now ready to conclude that there is finite set of points $P$ such that whenever $y \in \tw{\alpha}$ is given,  for each $i\in[3]$ there is a point $p \in P$ with $d(p, y + e_i) \leq \lambda D$. Here $P$ consists of $N(x_0), x_i + e_j$ and $x_{i,j} + e_k$ for suitable induces $i \not = j \not = k \not = i$, in particular $|P| = 15$.

\begin{lemma}\label{lemma4}Suppose we are given a 3-way set $\tw{\beta} = \cup_{i=1}^k A_i$ of diameter $C$, where diameters of sets $A_i$ are not greater than $\lambda r C$. Then there is constant $K_{k, r}$ (i.e. does not depend on $\lambda$ or $C$) such that $\tw{\beta}$ has a two-way subset of diameter at most $K_{k, r} \lambda C$. Further, we can take $K_{1, r} = r, K_{2,r} = 2 r + 8, K_{k+1, r} = K_{k, 2r+1}$ for all $r$ and $k \geq 2$.\end{lemma}

\begin{proof}[Proof of Lemma~\ref{lemma4}] We prove the Lemma by induction on $k$. When $k = 1$, there is nothing to prove, and $K_{1,r} = r$. Suppose $k = 2$.\\
Before we proceed, we need to establish:

\begin{lemma}\label{lemma5}Consider a 3-coloring of edges of complete graph $G$ whose vertex set consists of positive integers, namely $c: \{\{a,b\}: a\not=b, a,b \in \mathbb{N}\} \rightarrow [3]$. Then we can find sets $A, B$ whose union is $\mathbb{N}$, while for some colors $c_A, c_B$, we have $\diam_{c_A} G[A], \diam_{c_B} G[B] \leq 8$. (Here $\diam_{c_0}$ means diameter of graph induced by the color $c_0$.) Furthermore, we can assume $A$ and $B$ non-disjoint when $c_a \not=c_b$.\end{lemma}

\begin{proof}[Proof of Lemma~\ref{lemma5}] Let $x$ be any vertex. Define $A_i = \{a: c(a,x) = i\}$, for $i \in [3]$, the monochromatic neighbourhood of colour $i$ of $x$. We shall start by looking at sets $A_i$, if these do not yield the lemma, we shall look at similar candidates for $A, B$ until we find the right pair of sets. The following simple fact will play a key role: if $X, Y$ intersect and $\diam_c G[X], \diam_c G[Y]$ are both finite, then $\diam_c G[X\cup Y] \leq \diam_c G[X] + \diam_c G[Y]$.\\ 
 
Firstly, if any of the sets $A_i$ is empty, then taking $A_j \cup \{x\}$ and $A_k \cup \{x\}$ for the other two indices $j,k$ proves the lemma. Otherwise, we may assume that all $A_i$ are non-empty. The next idea is to try to `absorb' all the vertices into two of the sets $A_i$. To be more precise, let $B_{i,j} = \{a_i \in A_i: \forall a_j \in A_j, c(a_i, a_j) \not= j\}$ for distinct $i,j \in [3]$. Then, 
$$\diam_{i} \{x\} \cup A_i \cup (A_j \setminus B_{j,i}) \leq 4$$
for all distinct $i, j$ (which is what we meant by `absorbing vertices' above). Observe that if $\{i, j, k\} = [3]$ and $B_{j,i}$ and $B_{j,k}$ are disjoint, then $A_j \setminus B_{j,i}$ and $A_j \setminus B_{j,k}$ cover the whole $A_j$ so we can take $c_A = i, c_B = k$ and $A = \{x\} \cup A_i \cup (A_j \setminus B_{j,i}), B = \{x\} \cup A_k \cup (A_j \setminus B_{j,k})$. Hence, we may assume that $B_{j,i}$ and $B_{j,k}$ intersect, and that in particular these are non-empty.\\

Observe also that for $\{i,j,k\} = [3]$, if we are given $a_i \in B_{i,j}, a_j \in B_{j,i}$ then $c(a_i, a_j) \not= i, j$ so $c(a_i, a_j) = k$. This implies $\diam_k G[B_{i,j} \cup B_{j,i}] \leq 2$. We shall exploit this fact to finish the proof.\\

Now pick atbitrary $a_3 \in B_{3,1} \cap B_{3,2}$. If $c(a_1, a_3) = 3$ for some $a_1 \in B_{1,2}$, then $\diam_3(B_{1,2} \cup B_{2,1} \cup A_3 \cup \{x\}) \leq 5$ and $\diam_1(A_1 \cup (A_2 \setminus B_{2,1})\cup\{x\}) \leq 4$, so we are done. The same arguments works for $a_3$ and $B_{2,1}$, allowing us to assume that no edge between $B_{1,2} \cup B_{2,1}$ and $a_3$ is colored by 3. Therefore, since $a_3 \in B_{3,1} \cap B_{3,2}$, we have $c(B_{1,2}, a_3) = 2$ and $c(B_{2,1}, a_3) = 1$.\\

Recall that previously we tried to absorb the vertices of $A_1$ to $A_2$ to have a set of bounded diameter in colour 2, but this failed for the set $B_{1,2}$. Now, we have $c(B_{1,2}, a_3) = 2$, so we can once again try the same idea, by looking for an edge of colour 2 between $a_3$ and $A_1\setminus B_{1,2}$ (vertices of which are joined by an edge of colour 2 to something in $A_2$).\\
Suppose that $c(a_1, a_3) = 2$ for some $a_1 \in A_1 \setminus B_{1,2}$. Then $\diam_2(A_1 \cup A_2 \cup \{x\} \cup \{a_3\}) \leq 8$, and taking $A_3 \cup \{x\}$ for the other set, proves the lemma. Analogously, the lemma is proved if $c(a_2, a_3) = 1$ for some $a_2 \in A_2 \setminus B_{2,1}$.\\

Finally, since $a_3 \in B_{3,1} \cap B_{3,2}$, we may assume that $c(A_1 \setminus B_{1,2}, a_3) = 3$ and $c(A_2 \setminus B_{2,1}, a_3) = 3$. Observing that $\diam_3(B_{1,2} \cup B_{2,1}) \leq 2$ and $\diam_3 (\mathbb{N} \setminus B_{1,2} \setminus B_{2,1}) \leq 4$, completes the proof. \end{proof}

We refer to $\diam_c$ as the \emph{monochromatic diameter} for $c$.\\
Consider the complete graph on $\tw{\beta}$ along with a edge 3-colouring $c$, such that $x\cont{c(xy)}y$. Due to Lemma~\ref{lemma5}, we have sets $B_1, B_2$ whose union is $\tw{\beta}$, and their monochromatic diameters for some colors are at most 8, that is, by the triangle inequality $\diam (B_1 + e_{i_1}) \leq 8 \lambda C, \diam (B_2 + e_{i_2}) \leq 8 \lambda C$ for some $i_1, i_2$. If $i_1 = i_2$ we are done, hence we can assume these are different, and in fact without loss of generality $i_1 = 1, i_2 = 2$. If $A_1, A_2$ intersect, then diameter of union is not greater than $2r\lambda C$, proving the claim. Therefore, we shall consider only the situation when these are disjoint. Similarly, if $B_1 + e_1$ intersects both $A_1, A_2$, by triangle inequality, $\diam \tw{\beta} \leq (2r + 8) \lambda C$, so without loss of generality $B_1 + e_1 \subset A_1$. Depending on which of the two sets contains $B_2 + e_2$, we distinguish the following cases:
\begin{itemize}
\item \textbf{Case 1:} $A_1 \supset B_2 + e_2$.\\
We now claim that $A_1$ has a 2-way subset, whose diameter is then bounded by the diameter of $A_1$, which suffices to prove the claim. Suppose $a \in A_1$. Then $a \in B_i$ for some $i$, hence $a + e_1$ or $a + e_2$ is in $A_1$. If both are, there is nothing left to prove. Otherwise, the other point must be in $A_2$, say $a + e_1 \in A_1, a + e_2 \in A_2$. Suppose $a + e_3 \in A_2$ as well. Then $a + e_2 - e_1, a + e_3 - e_1 \in B_2$, thus $a + (-1, 2, 0),  a + (-1, 1, 1) \in A_1$, hence contracting $a, a - e_1 + e_2$ gives that $d(A_1, A_2) \leq \lambda C$. Otherwise $a + e_3 \in A_1$, hence we are done.
\item \textbf{Case 2:} $A_2 \supset B_2 + e_2$.\\
Color point by $i$ if it belongs to $A_i$. Such a colouring satisfies the hypothesis of the Lemma~\ref{strongcolLemma} since given a point $y$, either $y+e_1$ is colored by 1, or $y+e_2$ is colored by 2, and the second condition is also satisfied, (or after contraction we get $d(A_1, A_2) \leq \lambda C$ so done). Hence, we have a colouring that is essentially trivial, proving the claim.
\end{itemize}
Suppose the claim holds for some $k \geq 2$, and we have $k+1$ sets. As before, we can assume that these are disjoint and thus define coloring $c$, such that $y \in A_{c(y)}$. Further, we can assume that $d(A_i, A_j) > \lambda C$ for distinct $i, j$. Moreover, we have $A_i \cap \tw{\beta + (1,1,1)} \not=\emptyset$, as otherwise we are done by considering $\beta + (1,1,1)$ instead of $\beta$.\\
Let $z \in \tw{\beta}$. Define \emph{signature} of $z$ as $\sigma(z) = (c(z+e_1), c(z+e_2), c(z+e_3))$. By the discussion above given $i \in [k +1], l\in[3]$ we have a point $z$ such that $\sigma(z)^{(l)} = i$. Also, whenever $z, z'$ are two points in our 3-way set, we must have $\sigma(z)^{(i)} = \sigma(z')^{(i)} $ for some $i$, for otherwise we violate the condition on the distance between the sets $A_j$.\\
Let $(a,b,c)$ be a signature. Suppose there was another signature $(p, d, e)$, where $b \not= d, c \not= e$, which implies $p = a$. Since $k+1 \geq 3$, there are signatures $(g_1, h_1, j_1), (g_2, h_2, j_2)$, where $g_1, g_2, a$ are distinct. Then $(h_1, j_1) = (h_2, j_2) \in \{(b,e), (d,c)\}$, without loss of generality these are $(b,e)$. Hence, for any $z$ we have $\sigma(z)^{(2)} = b$ or $\sigma(z)^{(3)} = e$. Now, define a new colouring $c'$ of $\tw{\beta}$, if a point $p$ was coloured by $b$ set $c'(p) = 1$, if it was coloured by $e$ set $c'(p) = 2$ otherwise $c'(p) = 3$. Recalling the previous observations we see that $c'$ satisfies the necessary assumptions in the Lemma~\ref{strongcolLemma}, and apply it (formally change the coordinates first) to finish the proof.\\
Otherwise, any two signatures must coincide at at least two coordinates. In particular, only possible ones are $(\cdot,b,c), (a,\cdot,c), (a,b,\cdot)$, where instead of a dot we can have any member of $[k+1]$. If $a\not=b,c$, we have that $\sigma(z + (1,0,-1)) = (a,b,a)$ and $\sigma(z + (1,-1,0)) = (a,a,c)$. Thus $\sigma(z + (2,-1,-1))^{(2)} = \sigma(z + (2,-1,-1))^{(3)} = a$, which is impossible. Similarly $b \in \{a,c\}, c\in\{a,b\}$ hence $a=b=c$, and so $A_a$ is a two-way set with the wanted diameter.\end{proof}

By Lemma~\ref{lemma4}, there is a 2-way set $T$ with $\diam T \leq K_{15, 2}\lambda D$. Setting $s = 3$, we have $K_{15, s-1} = K_{14, 2s-1} = K_{13, 2^2 s-1} = \dots = K_{2, 2^{13}s - 1} = 2^{14}.3 + 6 = 49158$, as wanted.\end{proof}

We say that a set of points of the grid $Q$ is a \emph{quarter-plane} if there are distinct $i_1,i_2 \in [3]$ such that $Q = \{t + a e_{i_1} + b e_{i_2}: a,b \in \mathbb{N}_0\}$, for some point $t$.

\begin{proposition}\label{2wayqpProp}Suppose $\lambda < 1/4$ and there is a 2-way set $S$ of diameter $D$. Provided $m_1 = \inf_{s\in S} \rho(s) > D(2 + \lambda)$, $S$ contains a quarter-plane subset $Q$.\end{proposition}

\begin{proof} Without loss of generality, we can assume that $S$ has a point $p$ such that $S\subset \tw{p}$, and all points $s$ of $S$ except $p$ have a unique point $s'$ such that $s \in N(s')$. This is because we can always pick such subset of $S$, and it suffices to prove the statement in such a situation. We say that such a $k$-way set is \emph{spreading} (from $p$).
\begin{itemize}
\item\emph{Case 1}. For all $i \in [3]$, there is $x$ with $x+e_i$ not in $S$. \\
Let $x, y \in S$ be points such that $x+e_i, y+e_j \not\in S$, for $i,j$ distinct. Take $k$ so that $\{i,j,k\} = [3]$. Then if $x\cont{i}y$, by triangle inequality we have $m_1 \leq d(x, x + e_i) \leq d(x,y) + d(y, y +e_i) + d(y+e_i, x+e_i) \leq (2+\lambda) D$, which is contradiction. Similarly we drop the possibility of $x\cont{j} y$ happening, hence $x \cont{k} y$. Hence, if we define $A_l = \{s\in S: s = t +e_l$ for some $t \in S\}$, these are all of diameter $\leq 2 \lambda D$.\\
Suppose $A_1$ and $A_2$ are disjoint. Consider $x$ such that $x + e_3 \notin S$. If $x + e_1 + e_2 \in S$, it is both in $A_1, A_2$, which is impossible. Hence, we have that $x+e_1+e_3, x+e_2+e_3 \in S$, thus $x+e_3 +e_1 + e_2$ is not in $S$, so we can repeat the argument, to get all the $x + (1,0,n)$ and $x + (0,1,n)$ in $S$. Now, by triangle inequality, we must have $x + (1,0,n) \cont{3} x + (0, 1, n), x + (1,0,n) \cont{3} x + (0, 1, n+1)$, for all non-negative $n$, so $(x + (1,0,n))_{n \geq 1}$ is Cauchy, which is contradiction. Thus $A_1, A_2$ intersect, and similarly $A_1$ and $A_2$ intersect $A_3$, therefore, take $T$ to be union of these, which is thus 2-way (as every point of $S$ belongs to some $A_i$, except the starting one), and has $\diam T \leq 4\lambda D$.\\
\item\emph{Case 2}. Suppose that there is $i$ such that for any $x \in S$, $x + e_i$ is in $S$.\\
Without loss of generality, we assume $i=3$. Pick any $x_0$ in $S$ and set $a = (x_0)^{(3)}$. Thus, starting at $x_0$ we can form the sequence $(x_n)_{n\geq0}$ such that $\{x_{n+1}\} = S \cap \{x_n + e_1, x_n + e_2\}$. Suppose we have $x, y$ among these such that $x + e_1, y + e_2 \in S$. Hence, $x + (1,0,n), x + (0,0,n), y + (0,1,n), y + (0,0,n)$ belong to $S$ for all nonnegative $n$, thus $x + (0,1,n), y + (1,0,n)$ are never elements of $S$. Now, contracting pairs $x + (0,0,n), y + (0,0,n)$ and $x + (0,0,n+1), y + (0,0,n)$ gives 1-way Cauchy sequence as in the Case 1. If there are no such $x,y$ then we have that $S$ contains quarter-plane.\\
\end{itemize}
Therefore, if we ever get into Case 2, we are done. Hence, let $S_1 = S$, then by case 1, we have a 2-way $S_2$ subset of $S_1$, which we can assume to be spreading, by the same arguments as those for the set $S$. It also satisfies the necessary hypothesis of this claim, so we can apply the Case 1 once more to obtain 2-way set $S_3 \subset S_2$. Proceeding in the same manner, we obtain a sequence of spreading 2-way sets $S_1 \supset S_2 \supset \dots$, whose diameters tend to zero, so just pick a point in each of them, and then find a  1-way Cauchy sequence containing these to reach a contradiction.\end{proof}
\begin{proposition}\label{qpProp}Let $\{i_1, i_2, i_3\} = [3]$. Suppose we have a quarter-plane $S = \{\alpha + m e_{i_1} + n e_{i_2}: m,n \in \mathbb{N}_0\}$, of diameter $D$, and let $R = \inf_S \rho$. Provided $\lambda < 1/3$ and $D(1-\lambda^2) < (1-4\lambda)R$, there is a 3-way set of diameter at most $2\lambda(\frac{2}{1-\lambda}D + \frac{1+2\lambda}{1-\lambda}R)$.\end{proposition}

\begin{proof} Without loss of generality $i_3 = 1$. Observe that for any point $s \in S$ we must have $\rho(s) = d(s, s+e_1)$. The reason for this is that both $s+e_2, s+e_3 \in S$ and so $d(s, s+e_2), d(s, s+e_3) \leq \diam S = D$, but $\max \{d(s, s+e_1), d(s, s+e_2), d(s, s+e_3)\} = \rho(s) \geq R > D$.\\
Let $x_n \in S$ be a point with $\rho(x_n) < (1 + 1/n)R \leq 2R$. As $\lambda < 1/2$, we must have $x_n \cont{1} x_n + e_1$. Furthermore, suppose $i \not=1$ contracts $y, x_n + e_1$, for some point $y$ in $S$. Thus $x_n + e_i \in S$ and so $\rho(x_n+e_i) = d(x_n+e_i, x_n+e_1 + e_i)$. Then, by triangle inequality, we have $\rho(x_n+e_i) = d(x_n+e_i, x_n+e_1 + e_i) \leq d(x_n+e_i, y + e_i)+d(y+e_i, x_n+e_1 +e_i) \leq \lambda d(y, x_n + e_1) + D \leq \lambda 2R + (1 + \lambda) D < R$, therefore it must be $y \cont{1} x_n + e_1$. Hence $\rho(y) \leq d(y, x_n) + d(x_n, x_n + 2e_1) + d(x_n + 2e_1, y + e_1)\leq D + R(1 + 1/n) (1 + \lambda) + \lambda (D + R(1+ 1/n))$, for all $n$, hence $\rho(y) \leq D(1 + \lambda) + R(1+2\lambda) < 2R$.\\
Now we claim that for all $y \in S$, and all $k \geq 1$, we have $y \cont{1} y + k e_1$, which we prove by induction on $k$. For $k = 1$, we're done as otherwise there is $y$ with $\rho(y) < 2\lambda R < R$.\\
Suppose the claim holds for some $k \geq 1$. Then for any $y$ and $l \leq k+1$ we have $d(y, y + le_1) \leq d(y, y + e_1) + d(y + e_1, y + le_1) \leq \rho(y) + \lambda d(y, y + (l-1)e_1) \leq \dots \leq \rho(y) (1 + \lambda + \dots + \lambda^{l-1}) < \rho(y) / (1-\lambda)$. Also, $d(y, y + le_1) \geq d(y, y +e_1) - d(y + e_1, y + l e_1) \geq \rho(y) - \lambda d(y, y + (l-1)e_1) > \rho(y) \frac{1-2\lambda}{1-\lambda}$. As $\lambda < 1 - 2\lambda$, we have that 1 always contracts $y, y+(k+1)e_1$. In particular $\rho(y) \frac{1-2\lambda}{1-\lambda} < d(y, y +k e_1) < \rho(y) / (1- \lambda)$.\\
Fix any $x\in S$. Now, suppose $x \cont{i} y + ke_1$ for some $i \not = 1$. Then $R \frac{1-2\lambda}{1-\lambda}\leq \rho(y+e_i)\frac{1-2\lambda}{1-\lambda} < d(y + e_i, y + e_i + ke_1) \leq d(y + e_i, x + e_i) + \lambda (d(x, y) + d(y, y +ke_1)) \leq D(1 + \lambda) + \lambda \rho(y) / (1-\lambda)< (1+\lambda)D + \frac{2\lambda}{1-\lambda}R$, which is a contradiction. Hence, by looking at distance from $x + e_1$, we see that $\diam \{\alpha + (a,b,c): a \geq 2, b,c \geq 0\} \leq 2 \lambda(D + D(1 + \lambda)/(1-\lambda) + R(1+2\lambda)/(1-\lambda))$, as required. \end{proof}

In order to make the calculations throughout the proof easier, we use the following corollary instead.

\begin{corollary}\label{23Cor}Suppose we have a 2-way set $S$ of diameter $D$, and $R = \inf_{s\in S} \rho(s)$. Provided $\lambda < 1/9$ and $R > (2 + \lambda) D$, there is a 3-way set of diameter at most $6\lambda R$. \end{corollary}

\begin{proof} Firstly, apply the Proposition~\ref{2wayqpProp} to find a quarter-plane inside the given 2-way set. Since $R(1-4\lambda) > R/(2+\lambda) > D > (1-\lambda^2) D$ and $\lambda < 1/3$, we can apply the Proposition~\ref{qpProp}, to obtain a 3-way set of diameter at most $2\lambda(\frac{2}{1-\lambda}D + \frac{1+2\lambda}{1-\lambda}R)$. An easy calculation shows that this expression is smaller than $\lambda(5D + 3R) < 6\lambda R$.\end{proof}

Recall that we defined $\mu = \inf_x \rho(x)$, where $x$ ranges over whole grid. Recall also that $\mu > 0$ by the Proposition~\ref{mupos}.

\begin{proposition}\label{bounded3Prop}Given $K$, provided $1 > (2 + \lambda)\lambda K C_1$, all 3-way sets of have diameter greater than $K \mu$.\end{proposition}

\begin{proof} Clear for $K < 1$, so assume $K \geq 1$ and in particular $\lambda < 1/9$. Suppose contrary, let $T$ be a 3-way set of diameter $D \leq K\mu$. By the Proposition~\ref{32prop}, we know that there is a 2-way set $S \subset T$, with $\diam S \leq \lambda C_1 K \mu$. Therefore by Corollary~\ref{23Cor}, as $\lambda C_1 K \mu < \mu / (2 + \lambda)$, we have a 3-way set of diameter not greater than $6\lambda K\mu < \mu$, giving the contradiction. \end{proof}

\begin{proposition}\label{every2Prop}Given $K$, provided $\lambda < 1/9, 1/(3K)$, all 2-way sets have diameter greater than $\lambda K \mu$.\end{proposition}

\begin{proof} Suppose contrary, pick such a set $S_0$. Since $K \lambda \mu(2 + \lambda) < \mu$, we have a 3-way set $T_1$ with $r_1 = \diam T_1$ by Corollary~\ref{23Cor}. Now take a 2-way subset $S_1 \subset T_1$ with $\diam S_1 \leq K \lambda \mu$, so we have 3-way set $T_2$ of diameter not greater than $r_2 = 6\lambda r_1$. Repeating this argument, for all $k \geq 1$ we get that we can find 3-way set $T_k$, with diameter bounded by $r_k$, where $r_{k+1} = 6\lambda r_k$. But, then we must have $r_k < \mu$ for some $k$, giving contradiction.\end{proof}

Remark that the only way for a 2-way subset not to have elements in every $\tw{(n,n,n)}$ is to be contained in a union of finitely many quarter-planes.

\section{Finite contractive structures}
Recall the proofs of the Proposition~\ref{32prop} and Lemma~\ref{1contrLemma}. There we fixed a finite set $S$ of points, and then contracted various points with points in $S$ to obtain $k$-way sets. The following few claims pursue this approach further. In this subsection, we also show that we cannot have some configurations of points.  

\begin{proposition}\label{wtps} Suppose we have $K \geq 1$ and that $\lambda < 1/(24K)$ holds. Then we cannot have $x_0$ a point in the grid with $\rho(x_0) \leq K\mu$ such that $N(x_0) = \{x_1, x_2, x_3\}$, where $x_1, x_2, x_3$ satisfy $\diam (N(x_0) \cup \{x_i + e_j: i,j \in [3], i \not= j\}) \leq \lambda K \mu$.\\
\end{proposition}

When using this Proposition (and that is to obtain a contradiction in the proofs that are to follow), we say that we are applying the Proposition~\ref{wtps} to $(x_0; x_1, x_2, x_3)$ with constant $K$.
\begin{proof} Suppose we do have points described in the assumptions. By the Lemma~\ref{1contrLemma}, in each 3-way set we have a point $t$ such that $d(t, x_0) \leq 2K\mu$. Consider the contractions of $t$ with $x_0, x_1, x_2, x_3$; our main aim is to obtain a 2-way set of a small diameter and then to use the Proposition~\ref{every2Prop} to yield a contradiction.\\
Observe that from the assumptions of the proposition, for any $\{i, j, k\} = [3]$, we have $\max \{d(x_i, x_i + e_i), d(x_i, x_i + e_j), d(x_i, x_i + e_k))\} = \rho(x_i) \geq \mu > \lambda K \mu \geq \max \{d(x_i, x_i + e_j), d(x_i, x_i + e_k)\}$. Thus for all $i \in [3]$, $\rho(x_i) = d(x_i, x_i + e_i)$ holds.\\
Suppose first that $t\cont{i}x_i$ for all $i\in[3]$. Take $i$ so that $t\cont{i}x_0$. Then $\rho(x_i) = d(x_i, x_i + e_i) \leq d(x_i, x_0 + e_i) + d(x_0 + e_i, t + e_i) + d(t + e_i, x_i + e_i) \leq \lambda K \mu + \lambda d(x_0, t) + \lambda d(x_i, t) \leq 6\lambda K \mu < \mu$, which is impossible.\\
Thus, there are distinct $i,j \in [3]$ with $t\cont{j}x_i$. If $j$ was to contract $t, x_j$, we get $\rho(x_j) = d(x_j, x_j + e_j) \leq d(x_j, x_i + e_j) + d(x_i + e_j, t +e_j) + d(t +e_j, x_j + e_j) \leq \lambda K \mu + \lambda d(x_i, t) + \lambda d(t, x_j) \leq 7 \lambda K \mu < \mu$, which is impossible. Therefore, for some $k \not= j$, we have $t \cont{k} x_j$. In particular, $d(t + e_j, x_1 + e_2) \leq d(t + e_j, x_i + e_j) + d(x_i + e_j, x_1 + e_2) \leq \lambda d(t, x_i) + \lambda K \mu \leq 4\lambda K\mu$, and in similar fashion $d(t + e_k, x_1 + e_2) \leq 4\lambda K\mu$. Furthermore by the triangle inequality, both $t + e_j$ and $t + e_k$ are on the distance at most $K\mu + 4 \lambda K\mu \leq 2K\mu$ from $x_0$, so we can apply the same arguments to these points as we did for $t$. Hence, we obtain a bounded 2-way set of diameter at most $4K\mu$. But, considering all the points of the 2-way set except $t$ and their distance from $x_1 + e_2$, this is actually a 2-way set of diameter at most $8\lambda K\mu$, and we have such a set in every 3-way subset of the grid. Now, apply the Proposition~\ref{every2Prop} to obtain a contradiction, since $\lambda < 1/(24K)$ and $K\geq 1$. \end{proof}
\begin{proposition}\label{n1neighProp}Given $K\geq 1$, provided $\lambda < 1/(78K), 1/(13 C_1)$, there is no $x$ such that $\rho(x) \leq K\mu$, but $\rho(x+e_i) > 7 K \mu$ for all $i\in[3]$.\end{proposition}

Sometimes we refer to a pair of points $a,b$ in the grid as the \emph{edge} $a,b$, and by the \emph{length} of the edge $a,b$ we mean $d(a,b)$. The points $a$ and $b$ are the \emph{endpoints} of the edge $a,b$. 

\begin{proof} Suppose there was such an $x$. Consider the contractions of $x + e_i, x+e_j$ for $i\not = j$ and suppose that two such pairs are contracted by the same $k$. Thus $\diam\{x+e_k + e_1, x+e_k + e_2, x+e_k + e_3\} \leq 4\lambda K \mu$. Now, contract $x, x+e_k$ to get $\rho(x + e_k) \leq (2 + 5\lambda) K \mu < 3 K \mu$, giving us contradiction. So, the pairs described above must be contracted in different directions. Further, we can make a distinction between the \emph{short} edges of the form $a, a + e_i$ and the \emph{long} edges $a + e_i, a + e_j$, where $a$ is any point of the grid and $i,j$ are distinct integers in $[3]$\footnote{Note, short and long have nothing to do with the length of an edge previously defined, but actually just describe how these edges appears in the figures in the proofs.}. For every such long edge we have a unique short \emph{orthogonal} edge $a, a + e_k$ where $\{i,j,k\} = [3]$. We can observe that if we have a short edge and a long edge in $\{x\} \cup N(x)$ which are not orthogonal, but both contracted by some $i$, we must have another such pair, contracted by some $j\not=i$. One can show this by looking at the short edge $e$ which is orthogonal to the long one in a given pair of edges contracted by $i$.\\
If we write $[3] = \{i,j,k\}$, then $j$ contracts one long edge, and so does $k$. But now consider the described orthogonal short edge $e$. It cannot be contracted by $i$, for otherwise $\rho(x+e_i)$ is too small. Thus, it gives us another desired pair. Having shown this, we have two cases, with at least two such pairs (i.e. non-orthogonal short and long edge contracted in the same direction), or no such pairs.\\
\begin{figure}
\includegraphics[scale = 0.3]{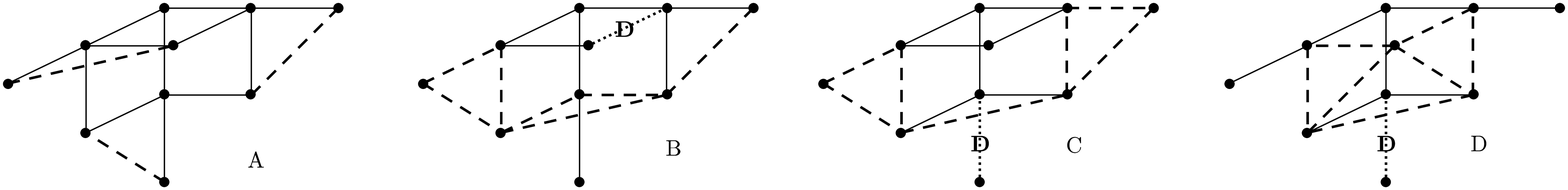}
\caption{Case 1}
\label{fig3}
\end{figure}
\emph{Case 1}. There are at least two such pairs.\\
In the Figure~\ref{fig3}, we show the possibilities for contractions, the edges shown as dash lines have length at most $3 K \lambda \mu$. Here we actually consider possible contractions and then apply triangle inequalities. This way, we obtain very few possible diagrams. We only list the possible configurations up to rotation or reflection, as the same arguments still carry through. In the diagram A, by short edge contractions we see that we get $\rho(x + e_i) \leq 3K\mu$ for some $i$, which gives the claim. Dot lines with letter \textbf{D} will be called the \emph{D-lines}. On the other hand, in the diagrams B, C and D, we claim that we are either done or the dot lines with letter \textbf{D} are of length at most $9\lambda K\mu$. Once this is established, we have $\rho(x + e_i) \leq 3K \mu$ for some $i$, resulting in a contradiction.\\
For each $i\in[3]$, let $x_i \in N(x)$ be such that $x_i +e_i$ is not an endpoint of long edge shown as dash line. By the Lemma~\ref{1contrLemma}, in each 3-way set we have a point $t$ with $d(x, t) \leq 2\rho(x)/(1-\lambda) \leq 3K\mu$. Observe that from the diagrams we have $d(x_i, x_i + e_j) \leq (2 + 6\lambda) K \mu$ whenever $i\not=j$. Further, we cannot have $x_i\cont{i}t$ for all $i$, otherwise we get a contradiction by considering contraction $x \cont{j} t$. If $x + e_j$ is an endpoint of a edge shown as a D-line, and $x + e_j + e_l$ is the other endpoint, we have $x + e_l = x_k$, hence $d(x + e_l + e_j, x + e_j) \leq \lambda (d(x + e_j, t) + d(t, x))\leq  7\lambda K\mu$, which is impossible. Thus, $x + e_j$ is not on a D-line edge, which gives $\rho(x_j) = d(x_j, x_j + e_j) \leq d(x_j, x) + d(x, x+e_j) + d(x+e_j, t + e_j) + d(t +e_j, x_j + e_j) \leq (2 + 7\lambda)K\mu$.\\
Previous arguments imply that we must have $i\not=j$ with $x_i\cont{j}t$, and hence $x_j\notcont{j}t$ (otherwise $\rho(x_j) \leq (2 + 14\lambda)K\mu$), so, given such a $t$, we get $t+e_a, t+e_b$, $a\not=b$ on distance at most $13\lambda K \mu$ from $x_1 + e_2$ and on distance not greater than $3K\mu$ from $x$, by the triangle inequality. Hence, in every $\tw{z}$ we get a 2-way subset of diameter not greater than $\lambda 26K \mu$, yielding a contradiction, due to $\lambda < \frac{1}{78K}$ and the Proposition~\ref{every2Prop}. Hence, edges shown as D-lines satisfy the wanted length condition.\\     

\begin{figure}
\includegraphics[scale = 0.4]{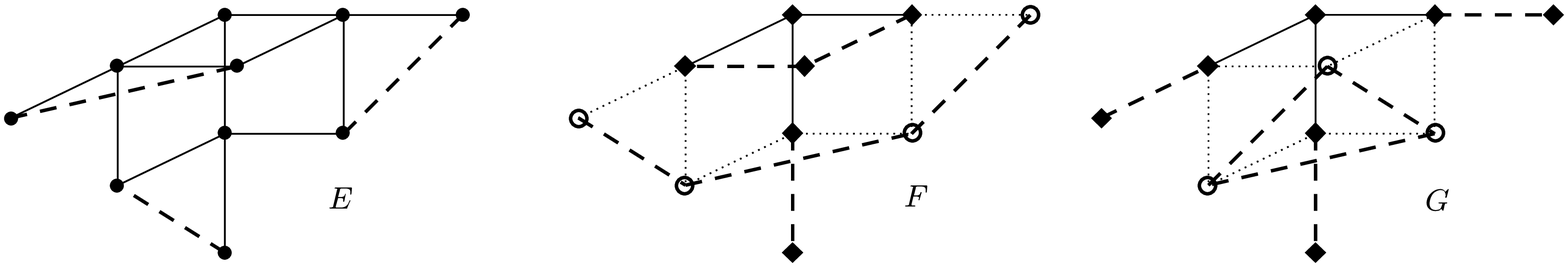}
\caption{Case 2}
\label{fig4s}\end{figure}

\emph{Case 2}. There are no such pairs. \\
The possible cases up to rotation or reflection are shown in the Figure~\ref{fig4s}, where the short edges shown as dash lines are of length at most $\lambda K \mu$, while the long ones are of the length $2\lambda K \mu$. As above, the diagram A gives $\rho(x+e_i) \leq 3 K \mu$, immediately. On the other hand, we can consider points shown as black squares and empty circles in the other two possibilities. We call a point \emph{black} if it is a black square and \emph{white} if it is shown as an empty circle. In the course of the proof, we shall colour more points in the black and white. Let $r$ be the minimal length of dotted edges in the Figure~\ref{fig4s}, and $r'$ the maximal. Then we have $r' \leq r + 2K\mu + 6\lambda K\mu$. Furthermore, given $i\in[3]$ we have $7K\mu < \rho(x+e_i) \leq r' < r + 3 K \mu$, so $r > 4K\mu$.\\
Consider $t$ such that $d(x,t) \leq 2r$. Let $j$ contract $x+e_i,t$, so we have $d(t+e_j, x) \leq d(t + e_j, x + e_i + e_j) + d(x + e_i + e_j, x + e_i) + d(x + e_i, x) \leq \lambda d(t, x + e_i) + \rho(x+e_i) +\rho(x) \leq \lambda(d(t,x) + d(x, x+e_i)) + r' + K\mu \leq 2\lambda r + \lambda K \mu + r + 2K\mu + 6\lambda K \mu + K\mu \leq (1 + 2\lambda) r + (3 + 8\lambda)K\mu < (1 + 2\lambda + \frac{3+8\lambda}{4}) r \leq 2r$, since $\lambda < 1/16$. Similarly if $j$ contracts $x, t$ we have $d(t + e_j, x) \leq d(t + e_j, x + e_j) + d(x + e_j, x) \leq 2\lambda r + K\mu \leq 2r$, as well. Further, observe that if $t + e_j$ is the result of contraction as before, then we have a point $a \in N(x) \cup \{x + e_i + e_j: i,j \in[3]\}$ with $d(t + e_j, a) \leq \lambda (2r + K\mu).$\\
Restrict our attention to the black (shown as black squares) and white (shown as empty circles) points shown in the Figure~\ref{fig4s}. We have $\diam\{$white points$\} \leq 6\lambda K \mu, \diam\{$black points$\}\leq (2 + 2\lambda) K\mu$ and the distance from any white to any black point is at least $r - K\mu - 4\lambda K \mu$. Take a point $t$ on the distance at most $2r$ from $x$ (note that by the Lemma~\ref{1contrLemma} such a point exists in every 3-way set). Consider contractions with $\{x\} \cup N(x)$ and suppose that $t + e_i, w$ and $t + e_i, b$ are results of these operations, where $w$ is a white and $g$ is a black point. Then, using the triangle inequality, we establish $r - K\mu - 4\lambda K \mu \leq d(w, b) \leq d(w, t + e_i) + d(t + e_i, b) \leq 2\lambda (2r + K\mu)$, which is a contradiction. For any given $i\in[3]$ let $x_i$ stand for the point of $N(x)$ such that $d(x_i, x_i + e_i) \leq \lambda K \mu$, thus $N(x) = \{x_1, x_2, x_3\}$. Let $t\cont{i}x$. Then take $j \in [3]$ distinct from $i$. We see that $x_j + e_i$ is white, while $x + e_i$ is black, hence $i$ does not contract $t, x_j$. Let $k \not= i$ contract $t, x_j$ and let $l$ be such that $\{i,j,l\} = [3]$. If $k = j$ then similarly we see that $x_l \cont{l} t$, while in the other case $k = l$ and $x_l \cont{j} t$. Hence, in conjunction with the previous arguments, we obtain a 3-way set of diameter at most $4r$.\\
Furthermore, recall that given pairs $t + e_i, p$ and $t + e_i, q$, which are results of contracting $t$ with $x$ or a point in $N(x)$, we must have $p$ and $q$ of the same color. As each of $t + e_1, t + e_2$ and $t + e_3$ is a result of such a contraction, we can extend the 2-colouring of the points in the diagrams F and G to all points of $\tw{x}$, namely $c:\tw{x}\to\{$black, white$\}$, with point $t + e_i$ being coloured by black, if $p$ described above is black in the original colouring, and white otherwise.\\
Now, the distance between any black point and any white point in the extended colouring is at least $r - K\mu - 4\lambda K \mu-2\lambda (2r + K\mu) = (1-4\lambda)r - (1 + 6\lambda)K\mu$. Recall the Proposition~\ref{32prop}, which guarantees the existence of a 2-way set $S \subset \tw{x}$ of diameter at most $4\lambda C_1 r$ from which we infer that $S$ is monochromatic, since $ 4\lambda C_1 r < (1-4\lambda)r - (1 + 6\lambda)K\mu$.\\

\emph{Case 2.1}. $S$ is black.\\
Consider any $t \in \tw{x}$ which has two black neighbors $t + e_{i_1}, t + e_{i_2}$, where $i_1 \not=i_2$. Then, letting $i_3$ be the third direction, that is $[3] = \{i_1, i_2, i_3\}$, we have $t \cont{i_3} x_{i_3}$, since the points of $N(x_{i_3})\setminus \{x_{i_3} + e_{i_3}\}$ are white. Hence, for any $t\in S$, we have that $N(t)$ is black. Furthermore, from the same arguments we see that $t \cont{i} x_i$ for all $i\in[3]$. Now, if $t$ is in $S$, and without loss of generality so are $t + e_1, t + e_2$, then $N(t + e_1), N(t + e_2)$ are black, so at least two elements of $N(t + e_3)$ are black too, implying that $N(t + e_3)$ is black. But, now looking at $t$ gives $t\cont{3} x_3$ and similarly, looking at $t +e_1, t +e_2, t +e_3$ tells us that 3 contracts points $t +e_1, t +e_2, t+e_3$ with $x_3$.\\
Let $s$ be the distance from such a $t$ from $x$. Then, for all $i\in[3]$, we have a black point $p$ in $\{x\} \cup N(x)$, which is contracted with $t$ by $i$, so that $p + e_i$ is black as well. Now, by the triangle inequality, we get $d(x, t + e_i) \leq d(x, p) + d(p, p + e_i) + d(p + e_i, t + e_i) \leq d(x,p) + d(p, p + e_i) + \lambda d(p,t) \leq \lambda d(t,x) + (1 + \lambda) d(x, p) + d(p, p + e_i) \leq \lambda s + (2 + \lambda) (K\mu)$. As in the proof of the Lemma~\ref{1contrLemma}, we see that there is $t\in S$, such that $d(t, x) < 3 K \mu$. From the estimates we have just made, we can see that $d(t + e_i, x) < 3 K \mu$ for all $i\in[3]$. Without loss of generality $t, t + e_1, t + e_2 \in S$. Recalling that this implies $d(t + e_3, x_3), d(t + e_3 + e_i, x_3) < 3 \lambda K\mu$ where $i$ takes all the values in $[3]$, shows that $\rho(t + e_3) < 6\lambda K \mu$, which is a contradiction.\\
\emph{Case 2.2}. $S$ is white.\\
If $t\in S$, then the point in $N(t)\setminus S$ is black, by contracting $t, x$. Hence, by the Proposition~\ref{23Cor}, we must have a 3-way set inside $\tw{x}$ of diameter at most $6\lambda r$, since $(1-4\lambda)r - (1 + 6\lambda)K\mu > (1 - 4\lambda) r - (1 + 6\lambda) r/4 > 2r/3 > (2 + \lambda) \lambda 4 C_1 r$, since $\lambda < 1/(13 C_1)$. But, such a set has at least one black point, so it must have only black points, and we have a contradiction as in Case 2.1. \end{proof}

\begin{proposition}\label{neighdiamProp}Given $K\geq 1$, provided $\lambda < 1/(41 K C_1)$, there are no $x$ with $\rho(x) \leq K\mu$ and $\diam N(x) \leq \lambda K \mu$.\end{proposition}

\begin{proof} Suppose we have such an $x$. We start by observing that two pairs of form $x+e_i, x+e_j$ cannot be contracted by the same $k$. Otherwise, since $\diam N(x) \leq \lambda K \mu$, after an application of triangle inequality, we also have $N(x + e_k) \leq 2 \lambda^2 K \mu$. Let $t$ be such that $x \cont{t} x + e_k$. Then $d(x+e_k, x + 2e_k) \leq d(x + e_k, x + e_t) + d(x+e_t, x+e_k + e_t) + d(x+e_k + e_t, x + 2e_k) \leq \diam N(x) + \lambda d(x, x+e_k) + \diam N(x+e_k) \leq \lambda K \mu + \lambda K\mu + 2\lambda^2 K \mu < 4 \lambda K \mu$. But then, for any $s \in [3]$, we have $d(x+e_k, x+e_k + e_s) \leq d(x+e_k, x+e_k + e_k) + \diam N(x + e_k) < 6 \lambda K \mu < \mu$, implying that $\rho(x + e_k) < \mu$, which is impossible.\\  
Thus, all three pairs of the form $x+e_i, x+e_j$ are contracted in different directions, hence we can distinguish the following cases (up to symmetry):\\
\begin{figure}
\includegraphics[scale = .3]{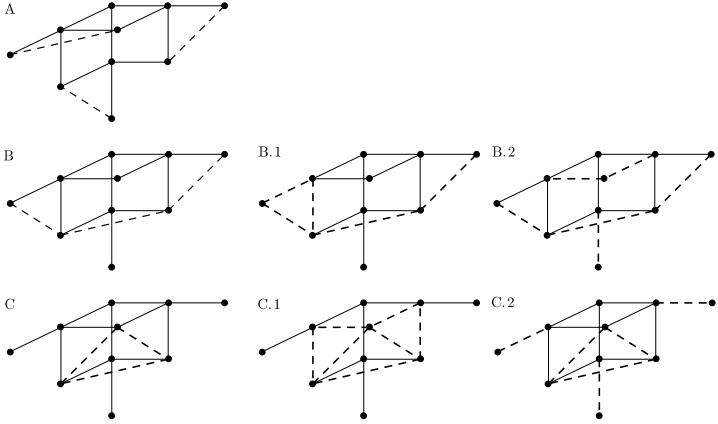}
\caption{Possible distances in the proof of the Proposition~\ref{neighdiamProp}}
\label{neighdiamfig}
\end{figure}
\emph{Case 1}. The results of contractions are shown as dash lines in the Figure~\ref{neighdiamfig}, diagram marked by A. It is not hard to see that after contracting pairs $x, x + e_i$, we get $\rho(x+e_j) < \mu$ for some $j$, giving us contradiction.\\
\emph{Case 2}. The results of contractions are shown as dash lines in the Figure~\ref{neighdiamfig}, diagram marked by B. By considering contractions of pairs $x, x+e_i$, we either get $\rho(x + e_j) < \mu$ for some $j$, or diagrams B.1, B.2 in figure~\ref{neighdiamfig}, where dash lines edges now indicate lengths at most $3\lambda K \mu$.\\
\emph{Case 3}. The results of contractions are shown as dash lines in Figure~\ref{neighdiamfig}, diagram marked by C. By considering contractions of pairs $x, x+e_i$, we either get $\rho(x + e_j) < \mu$ for some $j$, or diagrams C.1, C.2 in the Figure~\ref{neighdiamfig}, where now dash line implies length at most $3\lambda K \mu$.\\
Firstly, we will use the Proposition~\ref{wtps} to reject B.1 and C.1. In these two diagrams, for each $i \in [3]$, we can find unique $x_i \in N(x)$ such that $\rho(x_i) = d(x_i, x_i + e_i)$. Then we have $\diam (N(x) \cup \{x_i + e_j:i,j\in[3], i\not=j\}) \leq 15\lambda K \mu$. Also $\rho(x) \leq K\mu$, hence we can apply the Proposition~\ref{wtps} to $(x; x_1, x_2, x_3)$ with constant $15K$ to obtain a contradiction, since $\lambda < 1/(360K)$.\\
Observe that in the diagrams B.2 and C.2 we can denote $N(x) = \{x_1, x_2, x_3\}$ so that $d(x_i, x_i + e_i) \leq 3\lambda K \mu$. By the Proposition~\ref{n1neighProp}, we have that $\rho(x_i) \leq (7 + 7\lambda) K\mu$ holds for all $i\in[3]$, as $\lambda < 1/(78K), 1/(13C_1)$. Now, start from a point $t$ with $d(t, x) \leq 2\rho(x)/(1-\lambda) \leq 2K\mu /(1-\lambda) \leq 10K\mu$, which exists by the Lemma~\ref{1contrLemma}. Take any $p \in \{x\} \cup N(x)$ and contract with $t$. If $t \cont{i} p$, then $d(t + e_i, x) \leq d(t + e_i, p +e_i) + d(p+e_i, p) + d(p, x) \leq \lambda d(t,p) + d(p+e_i, p) + d(p, x) \leq \lambda(d(t,x) + d(x,p)) + d(p+e_i, p) + d(p, x) \leq \lambda 10 K \mu + (7 + 7\lambda) K\mu + (1+\lambda) K \mu \leq 10K\mu$.\\
Contract such a point $t$ with $x$ by some $i$. Write $[3] = \{i,j,k\}$ and consider contraction of $t, x_j$. It is not $i$ that contracts this couple of points, as otherwise $\rho(x_j) < \mu$. If it is $j$, then we can see that $x_k\cont{k}t$, and if it is $k$, then $x_k\cont{j}t$. Hence, all the points of $N(t)$ are on distance at most $10 K \mu$ from $x$, so we can repeat the argument to obtain a bounded 3-way set of diameter at most $20K\mu$. However, we get a contradiction by the Proposition~\ref{bounded3Prop}, since $1 > 41K C_1 \lambda$.\end{proof}

\begin{proposition}\label{tps} Given $K \geq 1$, suppose we have $x_0, x_1, x_2, x_3$ such that $\diam\{x_i + e_j:i,j \in [3], i\not=j\} \leq \lambda K\mu$. Furthermore, suppose $\rho(x_0) \leq K\mu$ and that $d(x_0, x_i) \leq K\mu$ for $i\in[3]$. Let $\{a,b,c\} = [3]$.\\
Provided $\lambda < 1/(820 C_1 K)$, whenever there is a point $x$ which satisfies $d(x + e_a, x + e_b) \leq \lambda K \mu$ and $d(x, x_0)\leq K\mu$, then we have $d(x + e_c, x_c+e_c) \leq 16 \lambda K \mu$.\end{proposition}

Remark that this is the Proposition~\ref{tpsO} in the overview of the proof. When using this Proposition, we say that we are applying the Proposition~\ref{tps} to $(x_0; x_1, x_2, x_3;x)$ with constant $K$. 

\begin{proof} Suppose contrary. Without loss of generality, we may assume $a = 1, b= 2, c =3$. Let us first establish $d(x, x+e_1) , d(x, x+e_2) \leq 3K\mu$. As $d(x + e_3, x_3 + e_3) > 16 \lambda K \mu$, we must have 1 or 2 contracting $x, x_3$. Similarly, we cannot have $x\cont{3} x_0$ and $x_0 \cont{3} x_3$ simultaneously. If $x\cont{3}x_0$ then we have $x_0\cont{i}x_3$ for some $i\in[2]$, and recall that $x\cont{j}x_3$ some $j\in[2]$, so $d(x, x+e_1) \leq d(x, x_0) + d(x_0, x_0 + e_i) + d(x_0 + e_i, x_3 + e_i) + d(x_3 + e_i, x_3 + e_j) + d(x_3 +e_j, x + e_j) + d(x + e_j, x + e_1) \leq K\mu + K\mu + \lambda K \mu + \lambda K\mu + 2\lambda K\mu + \lambda K\mu < 3K\mu$ and in the same way we get $d(x, x + e_2) < 3K\mu$. On the other hand if $x\cont{i}x_0$ for $i\in[2]$ we get $d(x, x + e_j) \leq d(x, x_0) + d(x_0, x_0 + e_i) + d(x_0 + e_i, x + e_i) + d(x+e_i, x + e_j) \leq K\mu + K\mu + \lambda K\mu + \lambda K\mu < 3K\mu$ for any $j\in[2]$.\\
Similarly, let us observe that $\diam \{x_1, x_2, x_3\} \cup \{x_i + e_j: i,j\in[3], i\not=j\} \leq 5 K\mu$. We see that this certainly holds in the case that there are distinct $i,j\in[3]$ with $x_0\cont{j}x_i$, as then $d(x_i, x_i + e_j) \leq d(x_i, x_0) + d(x_0, x_0 + e_j) + d(x_0 + e_j, x_i + e_j) \leq (2 + \lambda) K \mu$, and the claim about the given diameter follows. Hence, suppose that for all $i\in[3]$ contractions are $x_0\cont{i}x_i$. Then we cannot have $x_0 \cont{3} x$, so suppose that $x_0\cont{j}x$ and also that $x\cont{k} x_3$, where $j,k\in[2]$. Now, we can apply the triangle inequality to see $d(x_3 + e_k, x_3) \leq d(x_3 + e_k, x + e_k) + d(x + e_k, x + e_j) + d(x +e_j, x_0 +e_j) + d(x_0 + e_j, x_0) + d(x_0, x_3) \leq 2\lambda K \mu + \lambda K \mu + \lambda K \mu + K \mu + K \mu = (2 + 4\lambda) K \mu$, so once again we have the desired bound on the given diameter.\\
Now, by the Lemma~\ref{1contrLemma}, in every 3-way set we have a point $t$ with $d(t, x_0) \leq 7K\mu$. Suppose that for some distinct $i,j\in[3]$ we have $t\cont{i}x_i$ and $t\cont{i}x_j$. Then $d(x_i + e_i, x_i + e_k) \leq d(x_i + e_i, t + e_i) + d(t + e_i, x_j + e_i) + d(x_j + e_i, x_i + e_k) \leq 17\lambda K\mu$ for any $k\not = i$. Hence $\diam N(x_i) \leq 17\lambda K \mu$. However, contract $x_0, x_i$ to see that $\rho(x_i) \leq (2 + 18\lambda) K \mu < 17K\mu$. But we can apply the Proposition~\ref{neighdiamProp}, as $\lambda < (17.41 K C_1)$, to obtain a contradiction. Hence, we cannot have $x_i\cont{i}t$ and $x_j\cont{i}t$.\\
Suppose that for every such $t$ we have distinct $i,j \in[3]$ with $t\cont{i}x_j$. Then, by the previous observation, we see that $t\cont{k}x_i$, for some $k \not= i$. Hence $d(t + e_i, x_0) \leq d(t + e_i, x_j + e_i) + d(x_j + e_i, x_j) + d(x_j, x_0) \leq 8\lambda K\mu + 6 K\mu \leq 7K \mu$ and similarly for $t +e_k$. So, we can apply the same arguments to newly obtained points and proceeding in this manner we construct a bounded 2-way set. However, the points that we construct after $t$ are on the distance at most $9\lambda K \mu$ from $x_1 + e_2$, hence, we get a 2-way set of diameter at most $18\lambda K \mu$. This is a contradiction with the Proposition~\ref{every2Prop}, as we have such a point $t$ in every 3-way set and $\lambda < 1/(54K)$.\\
With this in mind, we see that in every 3-way set, there is a point $t$ with $d(x_0, t) \leq 7K\mu$ but for all $i\in[3]$ we have $t\cont{i}x_i$. Contract such a $t$ with $x$. It cannot be by 3, as then $d(x+e_3, x_3 + e_3) \leq 16\lambda K \mu$, so without loss of generality we have $x\cont{1}t$. But then for any $j\in\{2,3\}$ and $k\in[2]$ that contracts $x$ and $x_3$ we obtain $d(x_1 + e_1, x_1 + e_j) \leq d(x_1 + e_1, t + e_1) + d(t + e_1, x + e_1) + d(x + e_1, x + e_k) + d(x + e_k, x_3 +e_k) + d(x_3 + e_k, x_1 + e_j) \leq 8\lambda K\mu + 8\lambda K\mu + \lambda K \mu + 2\lambda K\mu + \lambda K \mu = 20 \lambda K\mu$, giving $\diam N(x_1) \leq 20\lambda K\mu$ and as before $\rho(x_1) \leq 20 K\mu$. Applying the Proposition~\ref{neighdiamProp} establishes the final contradiction, as $\lambda < 1/(820 C_1 K)$. \end{proof}

\section{Existence of certain finite configurations}

Our next aim is to show that, provided $\lambda$ is sufficiently small, certain finite configurations must exist. Recalling the Proposition~\ref{neighdiamProp}, we see that we are approaching the final contradiction in the proof of Proposition~\ref{keyprop}.

\begin{proposition}\label{ex1prop}Provided $\lambda < 1/(5\times10^{12})$, there is a point $x$ such that $\rho(x) \leq C_2 \mu$ and $\diam\{x, x + e_i, x+e_j\} \leq \lambda C_2 \mu$ for some distinct $i,j\in[3]$. Here $C_2 = 100000$.\end{proposition}

\begin{proof}
\begin{figure}
\includegraphics[scale=0.3]{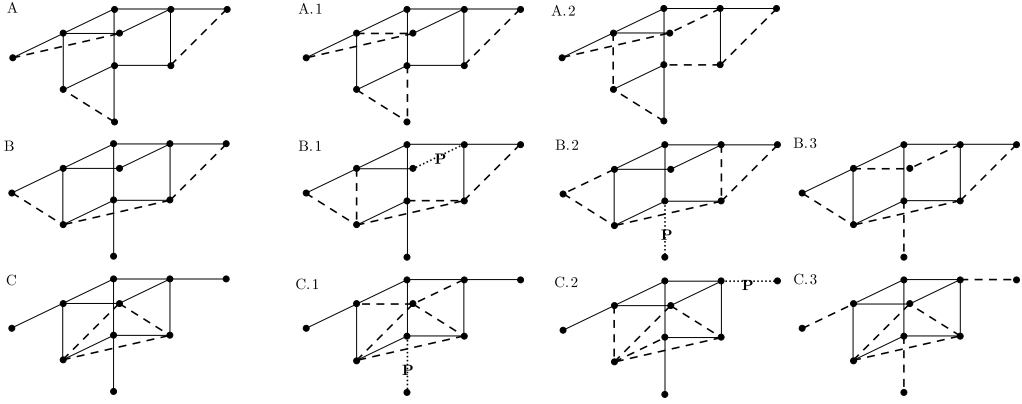}
\caption{Possible contractions in the proof of existence of auxiliary point}
\label{ex1prop1}
\end{figure}

Suppose contrary. The first part of the proof will be to establish the existence of an auxiliary point $y$ with $\rho(y) \leq 15 \mu$ and $d(y, y+e_i) \leq 192\lambda \mu, d(y+e_j, y + e_k) \leq 4\lambda\mu$ for some $\{i,j,k\} = [3]$. Pick any $t$ with $\rho(t) \leq 2\mu$ and consider contractions $\{t\} \cup N(t)$. As before, up to symmetry, we have the diagrams A, B and C in the Figure~\ref{ex1prop1} as possibilities for contractions of pairs of the form $t+e_a, t+e_b$, since no two such long edges can be contracted by the same $i$. If an edge is dash line in the Figure~\ref{ex1prop1}, then it is the result of a contraction of some pair of points in $\{t\}\cup N(x)$. As dot lines with letter \textbf{P} are shown the edges that will be results of applying the Proposition~\ref{tps}.\\
\emph{Case 1}. Suppose that we have diagram A. We see that we have diagrams A.1 and A.2 up to symmetry or otherwise some $\rho(z)$ is too small. However, diagram A.1 is impossible since $\rho(t + e_1) \leq C_2 \mu$ and $\diam\{t + e_1, t + e_1 + e_1, t + e_1 + e_2\} \leq \lambda C_2\mu$, which does not exist by the assumption. Hence, it is diagram A.2 that must occur, so we have $y$ with $\rho(y) \leq (4 + 6\lambda)\mu, d(y, y+e_3) \leq 2\lambda \mu, d(y+e_1, y +e_2) \leq 4\lambda\mu$.\\
\emph{Case 2}. Suppose that we have diagram B. As above, we can distinguish diagrams B.1, B.2, B.3, up to symmetry. First of all, if we have diagram B.3, we can apply the Proposition~\ref{n1neighProp} to $t$, as $\lambda < 1/(13 C_1), 1/(78.2)$, to obtain $\rho(t +e_i) \leq 14\mu$ for some $i$. Using this, we see that we have $\rho(y + e_3) \leq 15\mu, d(y + e_3 + e_1, y + e_3 + e_2) \leq 4\lambda\mu, d(y + e_3, y + 2e_3)\leq 2\lambda\mu$, as desired.\\
Consider now diagrams B.1 and B.2. We can denote $N(t) = \{t_1, t_2, t_3\}$ so that $t_1 + e_2, t_1 + e_3$ is a result of a contraction in $N(t)$ and so on. Observe that $\diam\{t_i + e_j:i,j\in[3], i\not=j\} \leq 12\lambda \mu$ and that $\rho(t) \leq 2\mu$, and in diagram B.1 $d(t + e_1, t +e_3) \leq 8\lambda \mu$, while in diagram B.2 $d(t + e_1, t + e_2) \leq 10\lambda \mu$, we can apply the Proposition~\ref{tps}, as $\lambda < (9840C_1)$, to $(t; t_1, t_2, t_3;t)$ with constant 12 to see that $d(t + e_2, t_2 + e_2) \leq 12\cdot 16\lambda\mu = 192\lambda\mu$ in diagram B.1 and $d(t + e_3, t_3 + e_3) \leq 192\lambda\mu$. Hence, $t + e_2$ in diagram B.1 and $t + e_3$ in the diagram B.2 are the desired points.\\
\emph{Case 3}. As in previous case, we are able to reach the same conclusion using the similar arguments.\\
To sum up, without loss of generality, we can assume that there is $y_0$, with $\rho(y_0) \leq 15\mu$, $d(y_0 + e_1, y_0 + e_2)\leq4\lambda\mu$ and $d(y_0, y_0 + e_3)\leq192\lambda\mu$. We shall now use this point to obtain a contradiction.\\
\begin{figure}
\label{ex1prop2}
\includegraphics[scale=0.3]{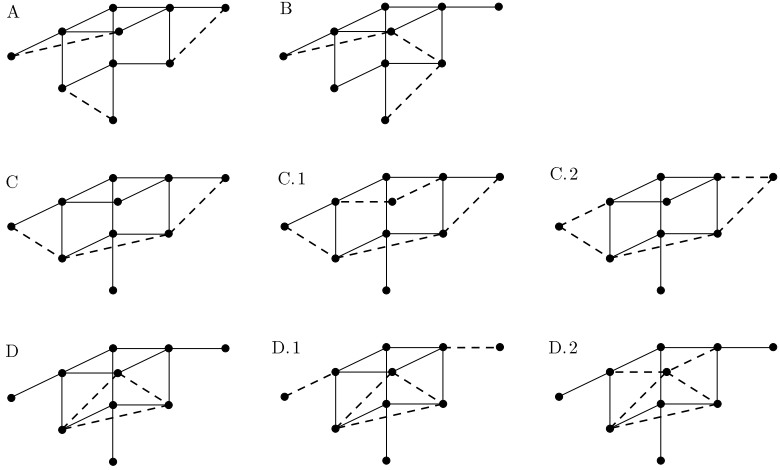}
\caption{Possible contractions in the neighborhood of an auxiliary point}
\end{figure}
Let $K = 20000$, and consider now those points which satisfy $\rho(y) \leq K\mu, d(y + e_i, y + e_j) \leq \lambda K\mu$ and $d(y, y+e_k) \leq \lambda K \mu$ for some $\{i,j,k\} = [3]$. We know that $y_0$ is one such point. Contract first the pairs inside $N(y)$, that is the long edges. As a few times before, it is not hard to see that for $i = 1, j = 2, k = 3$ we can only have the diagrams A, B, C and D Figure~\ref{ex1prop2} (if an edge is shown as dash line, that implies that it is a result of a contraction) and diagrams symmetric to these for different values of $i,j,k$. However, we can immediately reject diagram A, for if a point $y$ had diagram A, by contracting the short edges, we either obtain a point $t\in N(y)$ with $\rho(t) \leq 3K\mu$ and $\diam\{t, t + e_i, t + e_j\} \leq 3\lambda K \mu$, or we get a point $t\in N(y)$ with $\rho(y) \leq 4\lambda K \mu < \mu$, both resulting in a contradiction. Furthermore, if we were given a diagram B, then we can immediately apply the Proposition~\ref{tps} to $(y;y+e_3, y + e_2, y + e_1;y)$ with constant $6K$, as $\lambda < 1/(4920 C_1 K)$, which gives $d(y + e_3, y + e_1 + e_3) \leq 96\lambda K \mu$. Then we must have $y\cont{2}y  +e_3$, hence $\rho(y + e_1) \leq (1 + 97\lambda)K\mu < (K+1)\mu, \diam\{y + e_1, y + e_1 + e_1, y + e_1 + e_2\} \leq 5\lambda K \mu$ giving a contradiction once more.\\
Therefore, we must end up with either diagram C or D. Also observe that $y + e_i \cont{k} y + e_j$ must then hold for any $y$ that satisfies the properties stated above. Furthermore we must have $d(y + e_k, y + 2e_k) \leq 96\lambda K \mu$, as we can apply the Proposition~\ref{tps} to $(y; y_1, y_2, y_3; y)$, where $\{y_1, y_2, y_3\} = N(y)$ with constant $6K$. From this, we can conclude that neither $y \cont{k} y+e_i$ nor $y \cont{k} y+e_j$ can occur. Also we cannot have $y\cont{i}y+e_i$ and $y\cont{i}y+e_j$ simultaneously, as then $\rho(y + e_i) < \mu$, and similarly cannot have both $y\cont{j}y+e_i$ and $y\cont{j}y+e_j$. Hence, contracting the short edges implies that in fact we can only have the diagrams C.1, C.2, D.1 or D.2.\\
Observe that we can actually only have either C.1 and D.1, or C.2 and D.2 appearing. This is because if we had $y_1$ with the diagram among C.1 and D.1, and a point $y_2$ with a diagram among C.2 and D.2, we could first find the unique $e_i, e_j$ such that $d(y_1 + e_i, y_1 + e_i + e_1) \leq \lambda K\mu$ and $d(y_2 + e_j, y_2 + e_j + e_1) = \rho(y_2 + e_j)$. Now, apply the Proposition~\ref{tps} to $(y_1; y_1 + e_i, y_1 + e_k, y_1 + e_3; y_2+e_j)$ with constant $6K$, where $k \in[2]$ distinct from $i$, to obtain $\rho(y_2 + e_j) = d(y_2 + e_j, y_2 + e_j + e_1) \leq d(y_2 + e_j, y_2) + d(y_2, y_1) + d(y_1, y_1 + e_i) + d(y_1 + e_i, y_1 + e_i + e_1) + d(y_1 + e_i + e_1, y_2 + e_j + e_1) \leq K\mu + 2K\mu/(1-\lambda) + K\mu + \lambda K\mu + 96\lambda K\mu \leq 5K\mu$, while $\diam\{y_2 + e_j, y_2 + e_j + e_2, y_2 +e_j + e_3\} \leq 3\lambda K \mu$, which is a contradiction. Thus, we shall consider the cases depending on the pair of diagrams among these four that we allow.\\
\emph{Case 1}. We can only have diagrams C.1 and D.1.\\
Suppose that we had $y$ with $\rho(y) \leq K\mu/10, d(y + e_i, y + e_j) \leq \lambda K\mu/10, d(y, y +e_k)\leq \lambda K\mu/10$, for some $\{i,j,k\} =[3]$ that gave us diagram C.1 after contractions in $\{y\} \cup N(y)$. Without loss of generality, take $i = 1, j = 2$ and $k = 3$. Then, by the Proposition~\ref{n1neighProp} and the triangle inequality, we get $\rho(y + e_1), \rho(y + e_2) \leq K\mu$. In conjunction with $d(y + e_1 + e_1, y + e_1 + e_3), d(y + e_2 + e_2, y + e_2 + e_3) \leq \lambda K\mu/5$ and $d(y + e_1, y + e_1 + e_2), d(y + e_2, y + e_2 + e_1) \leq \lambda K \mu/10$, we see that $y + e_1, y + e_2$ are points whose neighborhoods contracting gives one of diagrams considered, in particular $y +e_1 + e_1 \cont{2} y+e_1+e_3$ and $y + e_2 + e_2 \cont{1} y + e_2+e_3$. But contract $y + e_1 + e_2$ with $y$, this gives $\rho(y + e_1 + e_2) \leq K\mu/5 < K\mu$ and $\diam N(y + e_1 + e_2) \leq \lambda^2 K \mu < \lambda K \mu$ which is a contradiction with Proposition~\ref{neighdiamProp}, since $\lambda < 1/(41 C_1 K)$.\\
Hence, as long as $y$ satisfies $\rho(y) \leq K\mu/10, d(y + e_i, y + e_j) \leq \lambda K\mu/10, d(y, y +e_k)\leq \lambda K\mu/10$, for some $\{i,j,k\} =[3]$ it must have diagram D.1. Start from $y_0$. Then we have $d(y_0 + e_3 + e_1, y_0 + e_3 + e_2) \leq \lambda ^ 2 K\mu, d(y_0 + e_3, y_0 + 2e_3) \leq \lambda^2K\mu$. Now, apply the Proposition~\ref{n1neighProp} to $y_0$ see that $\rho(y_0 + e_3) \leq 8 \rho(y_0)$. Therefore, contractions around $y_0 + e_3$ give us diagram D.1. But, contract $y_0 + e_1, y_0 + e_1 + e_3$ to obtain $\rho(y_0 + e_1 + e_3) <\mu$ or $\rho(y_0 + e_1) < \mu$.\\
\emph{Case 2}. We can only have diagrams C.2 and D.2.\\
Start from $y_0$ and define $y_n = y_0 + ne_3$ for all $n\geq 1$. By induction on $n$ we claim that $\rho(y_n) \leq 16\mu, d(y_n + e_1, y_n + e_2) \leq 4\lambda^{n+1}\mu, d(y_n + e_1, y_{n+1} + e_1) \leq (45 + 8n)\lambda^{n+1}\mu, d(y_n + e_2, y_{n+1} + e_2) \leq (45 + 8n)\lambda^{n+1}\mu, d(y_n, y_n + e_3) \leq 2000 \lambda \mu$.\\
For $n = 0$ the claim holds, since $y_0$ has diagram C.2 or D.2. Suppose the claim holds for some $n\geq 0$. Then it must have diagram C.2 or D.2, so $y_n + e_1\cont{3}y_n + e_2$, giving $d(y_{n+1} + e_1, y_{n+1} + e_2) \leq \lambda d(y_n + e_1, y_n + e_2) \leq 4\lambda^{n+1}\mu$. We can apply the Proposition~\ref{tps} to $(y_0;y_0 + e_2, y_0 + e_1, y_0 + e_3; y_n)$ or $(y_0;y_0 + e_1, y_0 + e_2, y_0 + e_3; y_n)$ (depending on the diagram of $y_0$) and to $(y_0;y_0 + e_2, y_0 + e_1, y_0 + e_3; y_{n+1})$ or $(y_0;y_0 + e_1, y_0 + e_2, y_0 + e_3; y_{n+1})$ with constant 60, so we get $d(y_0 + e_3, y_n + e_3), d(y_0 + e_3, y_{n+1} +e_3) \leq 960\lambda\mu$, thus $d(y_{n+1}, y_{n+1} + e_3) \leq 2000 \lambda\mu$. So $\rho(y_{n+1}) \leq (1+3\lambda)\rho(y_n) \leq 17\mu$, so $y_{n+1}$ has diagram C.2 or D.2. 
If the diagrams of $y_n$ and $y_{n+1}$ are distinct, then $y_n + e_1 \cont{3} y_{n +1 } + e_1$ and $y_n + e_2 \cont{3} y_{n +1} + e_2$, so the inequalities for $d(y_{n+1} + e_1, y_{n+2} + e_1)$ and $d(y_{n+1} + e_2, y_{n+2} + e_2)$ follow. Otherwise, $y_n + e_1 \cont{3} y_{n +1 } + e_2$ and $y_n + e_2 \cont{3} y_{n +1 } + e_1$, so $d(y_{n+1} + e_1, y_{n+2} + e_1) \leq d(y_{n+1} + e_1, y_{n+1} + e_2) + d(y_{n+1} + e_2, y_{n+2} + e_1) \leq 4\lambda^{n+2}\mu + \lambda (d(y_n + e_2, y_n + e_1) + d(y_n + e_1, y_{n+1} + e_1)) \leq 8\lambda ^{n+2}\mu + \lambda (45 + 8n)\lambda ^{n+1} = (45 + 8(n+1))\lambda^{n+2}\mu$. The inequality for $d(y_{n+1} + e_2, y_{n+2} + e_2)$ is proved in the same spirit.\\
Finally, by the triangle inequality we get $d(y_0 + e_1, y_{n+1} + e_1) \leq  d(y_0 + e_1, y_1 + e_2) + d(y_1 + e_2, y_1 + e_1) + d(y_1 + e_1, y_2+e_2) + \dots + d(y_n + e_1, y_{n+1} + e_1) < 50\lambda\mu$. Also $d(y_0, y_{n+1}) \leq d(y_0, y_0 + e_3) + d(y_0 + e_3, y_n + e_3) \leq 192\lambda \mu + 960\lambda \mu = 1152\lambda\mu$. Combining these conclusions further implies $\rho(y_{n+1}) \leq 16\mu$, as desired. Having established this claim, we can see that $(y_n + e_1)_{n\geq 0}$ is a 1-way Cauchy sequence, which is the final contradiction in this proof. \end{proof}

\begin{proposition}\label{ex2prop}Set $C_3 = 240000000000, C_{3,1} = 19000000000$ and let $i,j\in[3]$ be distinct. If $\lambda < 1 / (7380 C_1 C_{3,1})$, there is $x$ such that $\rho(x) \leq C_3 \mu; d(x + e_i, x + e_j) \leq \lambda C_3 \mu$.\end{proposition} 

\begin{proof} 

\begin{figure}
\label{ex2prop1}
\includegraphics[scale=0.3]{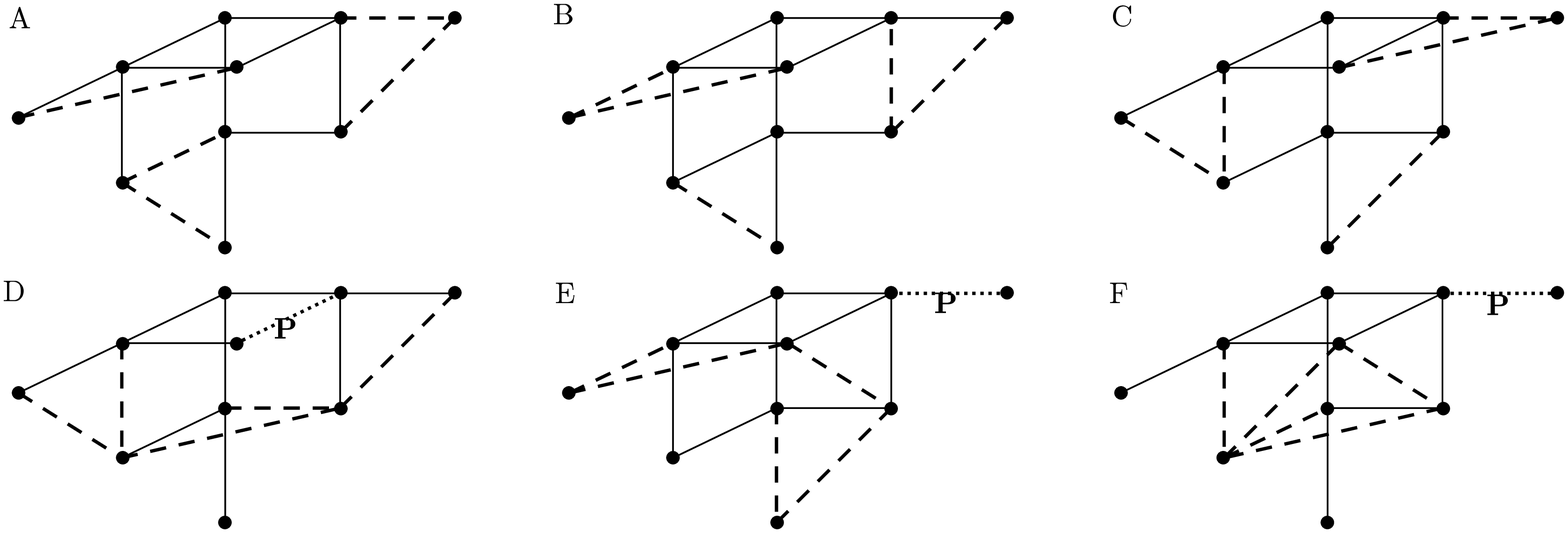}
\caption{Possible diagrams in the proof of the Proposition~\ref{ex2prop}}
\end{figure}

The proof will be a consequence of a few Lemmata, the last one being Lemma~\ref{ex2B}. It suffices to prove the claim for $i = 1, j = 2$. Suppose contrary, there is no such a point. Consider those $y$ which satisfy $\rho(y) \leq C_{3,1}\mu$ and $d(y + e_3, y  +e_i) \leq \lambda C_{3,1}\mu$. For such a point $y$ say that it is \emph{$C_{3,1}$-good}, and more generally use this definition for arbitrary constant instead of $C_{3,1}$. We already know that such a $y$ exists by the Proposition~\ref{ex1prop}. We list the possible diagrams of contractions in $\{y\} \cup N(y)$ for such a point, these are given in the Figure~\ref{ex2prop1} for $i = 1$. If an edge is shown as dash line, then it is a result of a contraction. Furthermore, with dot lines with letter \textbf{P} we mark edges whose length will be the result of applying the Proposition~\ref{tps}. It is not hard to show that these are the only possible diagrams, but for the sake of completeness we include an Appendix on the contraction diagrams, which in particular provides an explanation for the Figure~\ref{ex2prop1}. The symmetric diagrams to these for the case $i = 2$ are denoted by A', B', etc.\\
Our aim is to reject diagrams one by one. We shall start by discarding diagram A, and this method will then be used for the others. As we shall see, we can first apply the propositions proved so far to discard many diagrams in the presence of one given, and then the remaining ones can be fitted together so that we obtain a 1-way Cauchy sequence.\\
\begin{lemma}\label{ex2A}Set $C_{3,2} = 3100000000$. There is no $C_{3,2}$-good $y$ such that contractions give diagram A or A' for $y$.\end{lemma}
\begin{proof}[Proof of the Lemma~\ref{ex2A}] Suppose contrary, we do have such a point $y$, and without loss of generality $d(y + e_1, y + e_3) \leq \lambda C_{3,2}\mu$. Firstly, suppose that there was another point $z$ that is $C_{3,1}-$good, but whose diagram is among D, D', E, E', F, F'. By FNI we have $d(y,z) \leq (C_{3,2} + C_{3,1}) \mu/ (1-\lambda) < 2 C_{3,1} \mu$. Then, for a suitable choice $\{z_1, z_2,z_3\} = N(z)$, we can apply the Proposition~\ref{tps} to $(z;z_1,z_2,z_3;y + e_1)$ with constant $6C_{3,1}$ to get  $d(z_3, z_3 + e_3) \leq d(z_3, z) + d(z, y) + d(y, y + e_1 + e_3) + d(y + e_1 + e_3, z_3 + e_3) \leq C_{3,1}\mu + 2 C_{3,1}\mu + (1+\lambda)C_{3,2}\mu + 96\lambda C_{3,1}\mu < 4 C_{3,1}\mu$. Hence, $\rho(z_3) \leq 4C_{3,1}\mu$, except when the diagram is D or D', so we must apply the Proposition~\ref{n1neighProp} to $z$ first, so obtain $\rho(z_3) \leq 10 C_{3,1}\mu$. Also, $d(z_3+ e_1, z_3 + e_2) \leq 2\lambda C_{3,1}\mu$, but such a point $z$ cannot exist by the assumptions.\\
Now, take an arbitrary $(C_{3,1}/3)$-good point $z$ with diagram A. Consider the point $z + e_3$. We have $\rho(z + e_3) \leq (2+3\lambda)\rho(z), d(z + e_3 + e_1, z + e_3 + e_3) \leq 2\lambda \rho(z)$ so $z+e_3$ is $C_{3,1}-$good, so its diagram is one of A, B or C (it cannot be among the symmetric to these ones as then $\rho(z+e_3) < \mu$). If it was B, then contracting the pair $z + e_1, z + e_3 + e_2$ would give immediate contradiction, for we would obtain one of $\rho(z+e_1) < C_{3,1}\mu$, $\rho(z + e_2) < \mu$ or $\rho(z + e_2 +e_3) < \mu$. Similarly, it cannot be C, since contracting the same pair of points would give the contradiction once again as it would yield $\rho(z + e_2) < \mu$ or $\rho(z  +2e_3) <\mu$. Therefore, whenever we have a $(C_{3,1}/3)-$good point $z$ with diagram A, $z + e_3$ is $C_{3,1}-$good and has the same diagram.\\
Now, start from the $y$ given, and define $y_{n} = y + ne_3$, for $n\geq 0$. We shall now show that $(y_n)_{n\geq0}$ is a Cauchy sequence and hence obtain a contradiction. By induction on $n$ we claim $d(y_n, y_n + e_1) \leq \lambda^n C_{3,2}\mu, d(y_n + e_1, y_{n+1})\leq \lambda^{n+1} C_{3,2}\mu$, $\rho(y_n) < (2 + 10\lambda)C_{3,2}\mu$ and diagram of $y_n$ is A. This is clearly true for $n=0$.\\
Suppose that the claim holds for $n\geq0$. Note $d(y_0, y_{n+1}) \leq d(y_0, y_1) + d(y_1, y_2) + \dots + d(y_n + y_{n+1}) \leq C_{3,2}\mu + 2\lambda C_{3,2}\mu + 2\lambda^2 C_{3,2}\mu + \dots < (1 + 3\lambda) C_{3,2}\mu$. The fact that $y_n$ has the diagram A and is in fact $C_{3,1}/3-$good implies that $y_{n+1}$ is $C_{3,1}-$good and itself has diagram A. Further, $y_n\cont{3}y_n+e_1$ and $y_n+e_1\cont{3}y_n + e_3$. This is then sufficient to obtain the next two inequalities. Also $d(y_0 + e_2, y_{n+1}) < C_{3,2}\mu + 3C_{3,2}\mu/(1-\lambda) < 5C_{3,2}\mu$. Hence, we must have $y_0 + e_2 \cont{2} y_{n+1}$, for otherwise $\rho(y_0 + e_1) < \mu$ or $\rho(y_1) < \mu$. So $d(y_0, y_{n+1} + e_2) \leq d(y_0, y_0 + 2e_2) + \lambda d(y_0 + e_2, y_{n+1}) \leq C_{3,2}\mu + \lambda C_{3,2}\mu + 5\lambda C_{3,2}\mu$, from which we can infer $\rho(y_{n+1}) = d(y_{n+1}, y_{n+1} + e_2) \leq d(y_{n+1}, y_0) + d(y_0, y_0 + e_2) + d(y_0 + e_2, y_0 + 2e_3) + d(y_0 + 2e_2, y_{n+1} + e_2) < (2 + 10\lambda)C_{3,2}\mu$, as claimed.\\
From this we immediately get that $(y_n)_{n\geq0}$ is a Cauchy sequence. \end{proof}

\begin{lemma}\label{ex2E}Set $C_{3,3} = 1029000000$. There is no $C_{3,3}-$good point $y$ with diagram E or E'. \end{lemma}
\begin{proof}[Proof of the Lemma~\ref{ex2E}] Suppose contrary, without loss of generality $d(y + e_1, y + e_3) \leq \lambda C_{3,3} \mu$. Firstly, suppose there was a $C_{3,2}$-good point $z$ with $d(z + e_1, z + e_3) \leq \lambda C_{3,2}\mu$ a diagram among B, C, D. Take $p = z + e_2$ for diagrams B, D, $p = z + e_3$ for C, and apply the Proposition~\ref{tps} with constant $3C_{3,2}$ (for $d(p,y) \leq 3C_{3,2}\mu$ and for such a constant the other necessary assumptions also hold) to $(y;y+e_3, y + e_2, y + e_1;p)$ to obtain a contradiction at $y + e_1$, as it has $d(y + e_1  +e_1, y + e_1 +e_2) \leq 2C_{3,3} \lambda \mu$ and $\rho(y + e_1) = d(y + e_1, y + e_1 + e_3) \leq d(y + e_1, y) + d(y, z) + d(z, p + e_1) + d(p + e_1, y  +e_3 + e_1) \leq C_{3,3}\mu + (C_{3,3} + C_{3,2})\mu/(1-\lambda) + 2C_{3,2}\mu + 48 \lambda C_{3,2} < C_{3,1}\mu$. Hence, any such a $C_{3,2}-$good point $z$ can only have diagram E, or F.\\
Now, return to the point $y$ and define $y_n = y + ne_2$, for all $n\geq0$. We shall show that $(y_n)_{n\geq0}$ is a Cauchy sequence. By induction on $n$ we show that $d(y_n + e_1, y_n + e_3) \leq \lambda^{n+1} C_{3,3}\mu, d(y_n + e_3, y_{n+1} + e_3) \leq (3 + 2n)\lambda^{n+1} C_{3,3}\mu$ and $\rho(y_n) < 3C_{3,3}\mu$. Case $n=0$ is clear.\\
Suppose the claim holds for some $n \geq 0$. Firstly, $y_n$ is $C_{3,2}-$good, so it has diagram E or F, so in particular $d(y_{n+1} + e_1, y_{n+1} + e_3) \leq \lambda d(y_n + e_1, y_n + e_3) \leq \lambda ^ {n+1} C_{3,3}\mu$. Applying the triangle inequality gives $d(y_{n+1} + e_3, y_0 + e_3) \leq d(y_{n+1} + e_3, y_n + e_3) + \dots + d(y_1 + e_3, y_0 + e_3) \leq (5 + 2n)\lambda^{n+1} C_{3,3}\mu + \dots 3\lambda C_{3,3} \mu < 3\lambda C_{3,3} \mu/ (1-2\lambda)$. Further, since $d(y_0, y_{n+1}) \leq d(y_0, y_n) + d(y_n, y_{n+1}) \leq (\rho(y_0) + \rho(y_n))/(1-\lambda) + \rho(y_n) \leq 8 C_{3,3}\mu$ apply the Proposition~\ref{tps} to $(y; y + e_3, y+e_2, y+e_1;y_{n+1})$ with constant $8C_{3,3}$ which gives $d(y_1, y_{n+2}) \leq 128 \lambda C_{3,3}\mu$. Therefore, $\rho(y_{n+1}) < 3 C_{3,3}\mu$, in particular is $C_{3,2}$-good, hence its diagram can also only be E or F. If $y_n$ and $y_{n+1}$ have the same diagram, then contract $y_n + e_1, y_{n+1} + e_3$, otherwise $y_n + e_3, y_{n+1} + e_3$. These must be contracted by 2, so using the triangle inequality gives in the former case $d(y_{n+1} + e_3, y_{n+2} + e_3) \leq d(y_{n+1} + e_3, y_{n+1} + e_1) + d(y_{n+1} + e_1, y_{n+2} + e_3) \leq \lambda ^{n+2} C_{3,3}\mu + \lambda d(y_n + e_1, y_{n+1} + e_3) \leq \lambda ^{n+2} C_{3,3}\mu + \lambda (d(y_n+e_1, y_n+e_3) + d(y_n + e_3, y_{n+1} + e_3)) \leq 2\lambda ^{n+2} C_{3,3}\mu + \lambda d(y_n + e_3, y_{n+1} + e_3) \leq (5 + 2n) \lambda ^{n+2} C_{3,3}\mu$ as desired. In later case we are immediately done.\\
Furthermore this claim implies that $(y_n + e_1)_{n\geq0}$ is a Cauchy sequence, so we obtain a contradiction.\end{proof}

\begin{lemma}\label{ex2F}Set $C_{3,4} = 147000000$. There is no $C_{3,4}-$good point $y$ with diagram F or F'.\end{lemma}
\begin{proof}[Proof of the Lemma~\ref{ex2F}] Suppose contrary, there is such a point $y$ and without loss of generality we may assume $d(y + e_1, y + e_3) \leq \lambda C_{3,4} \mu$.\\
Suppose that we have a point $z$ that is $3C_{3,4}-$good with diagram F and that $d(z + e_1, z + e_3) \leq 3\lambda C_{3,4}\mu$, and that $z + e_2$ being $C_{3,3}-$good has diagram B, C or D. If it was B, we would immediately obtain a contradiction by contracting $z + e_1, z + 2e_2$, and if it was C, contracting $z + e_1, z + e_2 + e_3$, would once again end the proof, both giving a point $p$ with $\rho(p) < \mu$, so suppose that it was D. Apply the Proposition~\ref{tps} to $(z; z + e_1, z+e_2, z+e_3; z + e_2 +e_3)$  and to $(z; z + e_1, z+e_2, z+e_3; z + 2e_2)$ with constant $12C_{3,4}$. Now $z + e_1$ is $7C_{3,4}-$good, so it has diagram among B', C', D', F'. However $\diam\{z + e_1 + e_2, z + 2e_1 + e_2, z + e_1 + e_2 +e_3\}\leq 4\lambda C_{3,4}\mu$, so it must in fact be F'. Apply the Proposition~\ref{tps} to $(z; z + e_1, z+e_2, z+e_3; z + e_1 +e_3)$  with constant $12C_{3,4}$. Thus $z + e_3 \cont{3} z + 2e_3$. Write $r = d(z  +e_3, z + 2e_3)$, so we see that FNI implies $ r - \rho(z)\leq  d(z, z +2e_3) \leq \lambda (r + \rho(z))/(1-\lambda)$, but $r \geq C_3\mu$ and $\rho(z) \leq 3C_{3,4}\mu$ give contradiction.\\
Hence, whenever $z$ is a $3C_{3,4}-$good point with diagram F, $z +e_2$ is $C_{3,3}-$good and has the same diagram. Now $(y + ne_2)_{n\geq0}$ is Cauchy by the arguments from the proof of the Lemma~\ref{ex2E}, since there we allow both E and F as diagrams. \end{proof}

\begin{lemma}\label{ex2D} Set $C_{3,5} = 21000000$. There is no $C_{3,5}-$good point $y$ with diagram D or D'.\end{lemma}
\begin{proof}[Proof of the Lemma~\ref{ex2D}] Suppose contrary, there is such a point $y$ and without loss of generality we may assume $d(y + e_1, y + e_3) \leq \lambda C_{3,5} \mu$.\\
Now consider a point $3C_{3,5}-$good point $z$ with the diagram D and $d(z + e_1, z + e_3) \leq \lambda 3C_{3,5}\mu$. Since $z + e_1$ is $C_{3,4}-$good, it can only have diagram B, C or D. If it was not D, contract $z + e_2, z + e_1 + e_3$ for the sake of contradiction, namely, if it was B we would get $\rho(z + e_2) < \mu$ or $\rho(z + 2e_1) \leq C_3 \mu$, but $d(z + 2e_1 + e_1, z + 2e_1 + e_2) \leq 2\lambda C_{3,5}\mu$ and if it was C, we would obtain $\rho(z + e_1 + e_3) < \mu$ or $\rho(z + 2e_1) < \mu$. Hence, whenever $z$ has the given properties, $z + e_1$ has diagram D.\\
Now, return to $y$, and consider the sequence $y_n = y + ne_1$, for $n\geq0$. By induction on $n$, we show that $\rho(y_n) \leq 3C_{3,5}\mu, d(y_n, y_n + e_3) \leq \lambda^n C_{3,5}\mu$ and $d(y_n + e_3, y_{n+1}) \leq \lambda^{n+1}C_{3,5}\mu$. Clearly true for $n=0$.\\
Suppose that the claim holds for some $n\geq0$. Then $y_n$ is $3C_{3,5}-$good so it has diagram D. Hence, $y_n \cont{1} y_n + e_3$ and $y_{n+1} \cont{1} y_n + e_3$, which establishes two of the necessary inequalities. Also, by the triangle inequality $d(y_{n+1}, y_0) \leq C_{3,5}\mu + 2\lambda C_{3,5}\mu / (1-\lambda)$, so we can apply the Proposition~\ref{tps} to $(y_0; y_0 +e_2, y_0 + e_1, y_0 + e_3; y_{n+1})$ with constant $6 C_{3,5}$ to get $d(y_{n+1} + e_2, y_0 + e_1 + e_2) \leq 96 C_{3,5}\mu$, in particular $\rho(y_{n+1}) \leq 3C_{3,5}\mu$, as desired.\\
Now it follows that $(y_n)_{n\geq0}$ is a Cauchy sequence.\\
\end{proof}

\begin{lemma}\label{ex2C} Set $C_{3,6} = 3000000$. There is no $C_{3,6}-$good point $y$ with diagram C or C'.\end{lemma}

\begin{proof}[Proof of the Lemma~\ref{ex2C}] Suppose contrary, there is such a point $y$ and without loss of generality we may assume $d(y + e_1, y + e_3) \leq \lambda C_{3,6} \mu$.\\
Firstly, suppose that we have a $3C_{3,6}$-good point $z$ such that $d(z + e_1, z + e_3) \leq 3\lambda C_{3,6}\mu$, and $z + e_1$ has diagram B. We shall obtain a contradiction by considering contractions in such a situation. First of all we can observe that $z + e_3 \cont{3} z + e_1 + e_3$. Note that $d(z + 2e_1, z + 2e_1 + e_3) > C_3\mu$, so $d(z + e_3, z + 2e_3) \geq d(z + 2e_1, z + 2e_1 + e_3) - d(z + e_3, z + 2e_1) - d(z + 2e_3, z + 2e_1 + e_3) > C_3\mu - 24\lambda C_{3,6}\mu$.\\
\emph{Case 1}. Suppose that $z + e_3 \cont{2} z + 2e_3$.\\
We see that $z + e_2 + e_3, z + 2e_3$ is not contracted by 1, and from FNI, we must have $\rho(z + 2e_3) \geq (1-\lambda) d(z, z+2e_3) - \rho(z) \geq (1-\lambda) d(z+e_3, z +2e_3) - (2-\lambda) \rho(z)$, thus $z + e_2 + e_3, z + 2e_3$ is neither contracted by 3, hence  $z + e_2 + e_3 \cont{2} z + 2e_3$. Now suppose that $z + 2e_2 \cont{3} z + e_1 + e_2$. Then $d(z + e_3, z  +2e_3) \leq d(z + e_3, z + e_2 +2e_3) + d(z + e_2 +2e_3, z + e_2 + e_3) + d(z + e_2 + e_3, z + 2e_3)$ so $d(z + e_3, z + 2e_3) (1-\lambda) \leq 3\rho(z)$ which is impossible.\\ 
Therefore we must have $z + 2e_2 \cont{2} z + e_1 + e_2$ and $z + 2e_1 \cont{3} z + 2e_2$, otherwise $\rho(z + e_1 + e_2) < \mu$. Finally, contract $z + 2e_1 $with $ z + 2e_3$ to get $\rho(z + 2e_1) < \mu$ or $\rho(z + 2e_3) < \mu$.\\
\emph{Case 2}. Suppose that $z + e_3 \cont{3} z + 2e_3$.\\ 
By FNI applied to $z, z+2e_3$ we see that $\rho(z + 2e_3) \geq (1-\lambda) d(z+e_3, z+2e_3) - (2-\lambda) \rho(z)$, hence $\rho(z + 2e_3) = d(z + 2e_3, z + e_2 +2e_3) \geq (1-\lambda) d(z+e_3, z+2e_3) - (2-\lambda) \rho(z)$. So we have $z +2e_3 \cont{2} z + e_2 + e_3$. Also $z + 2e_2 \cont{2} z + e_3$, from which we see that $z + 2e_2 \cont{2} z + 2e_1$ giving contradiction.\\
Thus, whenever we have a point $z$ as described, we must have $z+e_1$ with diagram C as well. Now, set $y_n = y + ne_1$ for $n \geq 0$. By induction on $n$ we claim that $d(y_n, y_n  +e_3) \leq \lambda^n C_{3,6}\mu, d(y_n+e_3, y_{n+1})\leq \lambda ^{n+1} C_{3,6}\mu, \rho(y_n) \leq 3C_{3,6}\mu$ and $y_n$ has diagram C. This is clear for $n=0$.\\
Suppose the claim holds for some $n\geq0$, so $y_n$ must have diagram C, from which the first two inequalities follow. Observe that $d(y_{n+1}, y_0) < \rho(y_0) + 2\lambda \rho(y_0)/(1-\lambda)$ and $d(y_0 + e_2, y_0 +e_2 + e_3) > C_3\mu$, so $y_{n+1}\cont{2}y_0 + e_2$. Therefore $\rho(y_{n+1}) < 3\rho(y_0) \leq 3C_{3,6}\mu$, which gives the rest of the claim, as $y_{n+1} = y_n + e_1$ must have diagram C, by the previous conclusions.\\
Hence $(y_n)_{n\geq0}$ is a 1-way Cauchy sequence, which is a contradiction. \end{proof}

\begin{lemma}\label{ex2B} Set $C_{3,7} = 100000$. There is no $C_{3,7}-$good point $y$ with diagram B or B'.\end{lemma}

\begin{proof}[Proof of the Lemma~\ref{ex2B}] Suppose contrary, there is such a point $y$ and without loss of generality we may assume $d(y + e_1, y + e_3) \leq \lambda C_{3,7} \mu$.\\
Consider a $6C_{3,7}-$good point $z$, which has $d(z + e_1, z + e_3) \leq 6\lambda C_{3,7}\mu$, which therefore must have diagram B. We have $\rho(z + e_2) \leq (2+3\lambda)\rho(z), d(z + e_2 + e_2, z + e_2 +e_3)\leq 2\lambda \rho(z)$, so $z + e_2$ is $C_{3,6}-$good so has diagram B'. Observe that $z + e_3\cont{3}z+e_2+e_3$ as $d(z+ (1,0,1), z + (1,1,1)) \geq R - 4\lambda \rho(z)$ and $d(z + (0,1,1), z + (0,2,1)) > C_3\mu - 2\lambda \rho(z)$, where $R = d(z + e_1, z + e_1 + e_3) > C_3\mu$. Also $z + e_1 \cont{3} z + e_3$ since $z$ has diagram B. Similarly, since $z+e_2$ has diagram B', we must have $z + 2e_2 \cont{3} z + e_2 + e_3$. Furthermore $\rho(z + e_1 + e_2) \leq (2+3\lambda)\rho(z + e_2) \leq (2+3\lambda)^2\rho(z), d(z + e_1 +e_2 + e_1, z + e_1 + e_2 + e_3) \leq 2\lambda \rho(z + e_2) \leq 5\lambda \rho(z)$, so $z + e_1 + e_2$ is $C_{3,6}-$good, hence itself has diagram B, from which we infer $z + (0,1,1) \cont{3} z + (1,1,1)$.\\
Suppose that $z + e_1 +e_3 \cont{3} z + e_2 + e_3$, so have $d(z +(1,0,2), z+(0,1,2)) \leq \lambda (R + 3\rho(z))$ and $d(z + (1,0,2), z + (0,0,2)) \leq \lambda (R + 6\rho(z))$. Thus $d(z + (1,0,1), z + (1,0,2)) \leq \lambda(R + 8\rho(z))$, hence $z \cont{3} z + 2e_3$, which implies $d(z + 2e_3, z + 3e_3) \geq R(1- \lambda) - 3\rho(z)$, so $z + e_1 +e_3\cont{1} z + 2e_3$ (if $z + e_1 +e_3\cont{2} z + 2e_3$, then $\rho(z + 2e_1 + e_2) < \mu$) giving $d(z + (2,0,1), z + (1,0,1)) \leq \lambda (R + 10\rho(z))$. Also $z + (1,1,0) \cont{1} z + (2,0,0)$ and $z + 2e_1 \cont{3} z + 2e_2$, but then contracting $z + 2e_1, z + 2e_3$, results in contradiction.\\
Thus $z + (1,0,1) \cont{1} z + (0,1,1)$, as otherwise $R(1-\lambda) \leq 2C_{3,7}\mu$, which is not possible. From the fact that $z$ has the diagram B, we have $z \cont{1} z + e_1$. Also, we must have $z + e_1 \cont{1} z + e_1 + e_2$. As $d(z + e_1 + e_3, z + 2e_1 + e_3) \geq (1-\lambda)R - 7\lambda \rho(z)$, we cannot have $z + e_1 \cont{3} z + 2e_1$. Suppose that $z + e_1 \cont{1} z + 2e_1$, then contracting $z + 2e_2, z + e_1$ and $z + 2e_2, z + 2e_1$ (both must be in the direction $e_3$) gives $d(z + e_1 + e_3, z + 2e_1 + e_3) \leq 6\lambda \rho(z)$, which is a contradiction.\\ 
We conclude that $z \cont{1} z + e_1$, $z+e_1 \cont{1} z + e_1 + e_2$ and $z+e_1 \cont{2} z + 2e_1$, for such a $z$. By symmetry, when $d(z + e_2, z + e_3) \leq 6\lambda C_{3,7}\mu$ holds instead of $d(z + e_1, z +e_3) \leq 6 \lambda C_{3,7}$, then we must have $z \cont{2} z + e_2$, $z+e_2 \cont{2} z + e_1 + e_2$ and $z+e_2 \cont{1} z + 2e_2$.\\
Return now to the point $y$ and consider the sequence given as $y_0 = y$, when $k$ is even set $y_{k+1} = y + e_2$, otherwise $y_{k+1} = y + e_1$. By induction on $k$ we obtain $\rho(y_k) \leq 3C_{3,7}\mu, d(y_{k}, y_{k+2}) \leq 3\lambda^{k} \frac{1+\lambda^ 2}{1-\lambda} C_{3,7}\mu, d(y_k, y_k + e_3) \leq \lambda^k C_{3,7}\mu$ and $d(y_k, y_k + e_1) \leq 3\lambda^k C_{3,7}\mu$ for even $k$, $d(y_k, y_k + e_2) \leq 3\lambda^k C_{3,7}\mu$ for odd $k$.\\
When $k = 0$, the claim clearly holds. Suppose that the claim is true for all values less than or equal to some even $k\geq0$. We shall argue when $k$ is even, the same argument works in the opposite situation. By the triangle inequality, we have $d(y_0, y_i) \leq 3\frac{1+\lambda^ 2}{(1-\lambda)(1-\lambda^2)} C_{3,7}\mu$ for even $i \leq k + 2$ and $d(y_1, y_i) \leq 3\lambda\frac{1+\lambda^ 2}{(1-\lambda)(1-\lambda^2)} C_{3,7}\mu$ for the odd $i \leq k +2$. In particular, as $y_k$ is $C_{3,6}-$good, it has diagram B, so $\rho(y_{k +1}) = d(y_{k+1}, y_{k+2}) \leq d(y_{k+1}, y_1) + \rho(y_0) + d(y_0, y_{k+2}) \leq 3(1+\lambda)\frac{1+\lambda^ 2}{(1-\lambda)(1-\lambda^2)} C_{3,7}\mu + C_{3,7}\mu \leq 5C_{3,7}\mu$ and $d(y_{k+1} + e_2, y_{k+1} + e_3) \leq 2\lambda \rho(y_k) \leq 10\lambda C_{3,7}\mu$. Then $y_{k+1}$ is $10C_{3,7}-$good, so it must have diagram B'. From the contractions implied by this diagram described previously, we get that $d(y_{k+1}, y_{k+1} + e_3) \leq \lambda^{k+1} C_{3,7}\mu$. Moreover, $y_{k+1} \cont{2} y_{k+1} + e_2$, $y_{k+1}+e_2 \cont{2} y_{k+1} + e_1 + e_2$ and $y_{k+1}+e_2 \cont{1} y_{k+1} + 2e_2$. Therefore $d(y_{k+1} + e_2, y_{k+3}) \leq d(y_{k+1} + e_2, y_{k+1} + 2e_2) + d(y_{k+1} + 2e_2, y_{k+1} + 2e_2 + e_1) + d(y_{k+1} + 2e_2 +e_1, y_{k+3}) \leq \lambda d(y_{k+1} + e_2, y_{k+3}) + (1+\lambda) d(y_{k+1} + e_2, y_{k+1} + 2e_2) \leq \lambda d(y_{k+1} + e_2, y_{k+3}) + \lambda(1+\lambda) d(y_{k+1}, y_{k+1} + e_2)$. Hence $d(y_{k+1}, y_{k+3}) \leq \frac{1 + \lambda ^2}{1-\lambda} d(y_{k+1}, y_{k+1} + e_2)$, proving the claim.\\
Now, we can infer that $y_0, y_0 + e_1, y_1, y_1 + e_1, y_2,\dots$ is a 1-way Cauchy sequence, which is a contradiction. \end{proof}

But now, the Proposition~\ref{ex1prop} provides us with a $C_{3,7}-$good point, which however cannot exist because of the Lemmata we have shown during this proof. \end{proof}

\section{Final contradiction}

In the remaining of the proof of the Proposition~\ref{keyprop}, an important role will be played by the sets $S_i(K,x_0) = \{y: d(x_0,y) \leq K\mu, d(y, y + e_i) \leq K\mu\}$, defined for any point $x_0$, constant $K$ and $i \in [3]$. Given any point $t$, the set $S_i(K, x_0)$ serves to give approximate versions of contractions of $x_0$ and $t$ in the direction $i$, in the following sense. If $t \cont{i} y$ for some $y \in S_i(K, x_0)$, then we have
\begin{multline*}d(x_0, t + e_i) \leq d(x_0, y) + d(y, y+e_i) + d(y+e_i, t+e_i) \leq K\mu + K\mu + \lambda d(y,t)\\
\leq 2K\mu + \lambda (d(y, x_0) + d(x_0, t)) \leq (2+\lambda)K\mu + \lambda d(x_0, t).\end{multline*}
Using this idea, unless $t$ never contracts with $S_i(K, x_0)$ in the direction $i$ for some $i$, we can get 3-way sets of small diameter, as we shall see in the proof of the next proposition.\\
An additional benefit of using these sets is that they usually do not only consist of $x_0$ (note $x_0 \in S_i(K, x_0)$ if $\rho(x_0) \leq K \mu$), and that for example, under certain circumstances, we can find a point $y$ with the property that $y, y +e_3 \in S_3(K, x_0)$. Such points will be used then in proving the Propositions~\ref{closecase} and~\ref{farcase}, which combined with the following proposition, finish the main proof of this paper.\\
Recall that $x \notcont{i} y$ means that $d(x + e_i, y + e_i) > \lambda d(x,y)$.
  
\begin{proposition}Fix arbitrary $x_0$ with $\rho(x_0) < 2\mu$. Given $K \geq 2$, when $i\in[3]$, define $S_i(K,x_0) = \{y: d(x_0,y) \leq K\mu, d(y, y + e_i) \leq K\mu\}$. Provided $1 > 2\lambda K C_1 (2+\lambda)^2 /(1-\lambda)$, in every $\tw{z}$ there is $t$ such that $d(t,x_0) \leq \frac{2+\lambda}{1-\lambda} K \mu$, but for some $i$ we have  when $s \in S_i(K,x_0)$. \end{proposition}

\begin{proof} First of all, we have $x_0 \in S_1(K, x_0),S_2(K, x_0),S_3(K, x_0)$, making these non-empty, as $K\mu \geq \rho(x_0) \geq d(x_0, x_0 + e_i)$ for all $i\in[3]$. Suppose contrary to our statement, there is $z$ without any $t$ described above. Since $\frac{2+\lambda}{1-\lambda} K \mu > \rho(x_0) / (1-\lambda)$, we know that there is $y \in \tw{z}$ such that $d(x_0, y) \leq \frac{2+\lambda}{1-\lambda} K \mu$, by the Lemma~\ref{1contrLemma}. Then we have $s_1 \in S_1$ such that $s_1\cont{1}y$. Hence $d(y+e_1, x_0) \leq d(y+e_1, s_1 + e_1) + d(s_1 + e_1, s_1) + d(s_1, x_0) \leq \lambda (d(y, x_0) + d(x_0, s_1)) + 2K\mu \leq \lambda (\frac{2+\lambda}{1-\lambda} K \mu + K \mu) + 2 K \mu = \frac{2+\lambda}{1-\lambda} K \mu$. Similarly, we get the same result for $y + e_2, y+e_3$, and so we have constructed a 3-way set of diameter not greater $2\frac{2+\lambda}{1-\lambda} K \mu$, but there are no these since $1 > 2\lambda K C_1 (2+\lambda)^2 /(1-\lambda)$ by the Proposition~\ref{bounded3Prop}, giving contradiction.\end{proof}

Similarly as before, we use tighter constraints on $\lambda$. Here we use $\lambda < 1/10$ implies $\frac{2+\lambda}{1-\lambda} < 3$ and $(2+\lambda)^2 /(1-\lambda) < 5$. Note that this is the Proposition~\ref{casescorO} described in the overview of the proof.

\begin{corollary}\label{casescor}Fix arbitrary $x_0$ with $\rho(x_0) < 2\mu$. Given $K \geq 2$, when $i\in[3]$, define $S_i(K,x_0) = \{y: d(x_0,y) \leq K\mu, d(y, y + e_i) \leq K\mu\}$. Provided $1 > 10\lambda K C_1$, in every $\tw{z}$ there is $t$ such that $d(t,x_0) \leq 3K\mu$, but for some $i$ we have $s \notcont{i} t$ when $s \in S_i(K,x_0)$.\label{forbCor}\end{corollary}

Based on this, we shall reach the final contradiction in the proof of the Proposition~\ref{keyprop}. To do this, we shall consider the possible cases on the $d(t + e_j, t + e_k)$ where $\{i,j,k\} = [3]$ and $t$ is a point given by the Corollary~\ref{forbCor}. Namely, suppose that $d(t + e_j, t + e_k)$ is small enough, and in fact $j = 1, k = 2, i = 3$. Then whenever we have a point $y$ with $y \in S_3(K, x_0)$ and if $d(y + e_1, y + e_2)$ is small we shall have $\diam\{y + e_1, y + e_2, t + e_1, t + e_2\}$ small as well. On the other hand, if $d(t + e_1, t + e_2)$ is large, and $y_1, y_2 \in S_3(K, x_0)$ with $d(y_1 + e_1, y_1 + e_2), d(y_2 + e_1, y_2 + e_2)$ small, but $d(y_1 + e_1, y_2 + e_1)$ large, we shall have pairs $t,y_1$ and $t,y_2$ contracted by the different values in $\{1,2\}$. Of course, we need to specify what do we mean by small and large in this context, and this is done in the following two propositions.

\begin{proposition}\label{closecase} Let $C_4 = 16C_3$. Fix $x_0$ with $\rho(x_0) < 2\mu$. Let $\{i,j,k\}=[3]$. Given $K$, provided $\lambda < 1/(44C_3 + 6C_4 + K), 1/(34440 C_1 C_3)$, we have $d(t + e_j, t + e_k) > K\lambda \mu$, when $t$ is such that $d(t, x_0) \leq 3 C_4 \mu$ and $s\notcont{i}t$ when $s \in S_i(C_4, x_0)$. \end{proposition}

\begin{proposition}\label{farcase} Let $C_5 = 1000 C_3$. Fix $x_0$ with $\rho(x_0) < 2\mu$. Let $\{i,j,k\}=[3]$. Provided $\lambda < 1/(8200000 C_1 C_3)$, we have $d(t + e_j, t + e_k) \leq 10C_5\lambda \mu$, when $t$ is such that $d(t, x_0) \leq 3 C_5 \mu$ and $s\notcont{i}t$ when $s \in S_i(C_5, x_0)$.\end{proposition}

Once we have shown these propositions, we just need to take $\lambda$ small enough so that they both hold. Now, let us prove a lemma that classifies the relevant possible diagrams which will then be used for arguing afterwards. Once that is done, we proceed to establish the propositions.

\begin{figure}
\includegraphics[scale = 0.3]{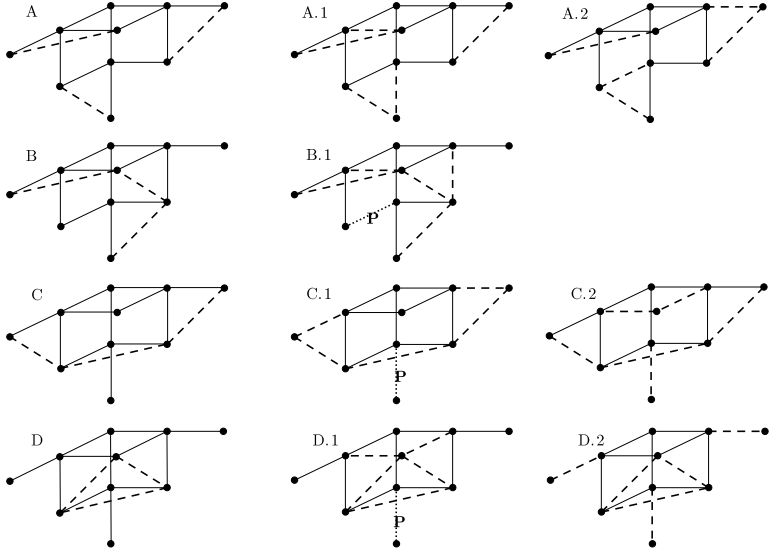}
\label{diagfig}
\caption{Possible diagrams for $\rho(y) \leq K\mu, d(y + e_1, y + e_2) \leq \lambda K\mu$.}
\end{figure}

\begin{lemma}\label{diaglemma} Let $K \geq 1$ and $\lambda < 1/(4920 K C_1)$. Suppose that we have a point $y$ with $\rho(y) \leq K\mu$ and $d(y + e_1, y + e_2) \leq \lambda K\mu$. Then $y$ must have one of the diagrams shown in the Figure~\ref{diagfig} (up to symmetry). \end{lemma}
\begin{proof}[Proof of the Lemma~\ref{diaglemma}]
Contracting the long edges in $N(y) \cup \{y\}$ can only, up to symmetry, give us diagrams A, B, C and D, as described in the first part of the Appendix, with the requirement $1/(164 K C_1) > \lambda$. Observe that in B, C and D, we can apply the Proposition~\ref{tps} to $(y;y+e_3, y+e_2, y + e_1; y+e_1), (y;y+e_2, y+e_1, y + e_3; y+e_3)$ and $(y;y+e_1, y+e_2, y + e_3; y+e_3)$ respectfully with constant $6K$, as long as $\lambda < 1/(4920 K C_1)$. Further, by contracting the short edges, we can only obtain diagrams A.1, A.2, B.1, etc. shown in the Figure~\ref{diagfig}, up to symmetry, as otherwise we obtain a point $p \in \{y\}\cup N(y)$ with $\rho(p) < \mu$. 
\end{proof}

\begin{proof}[Proof of the Proposition~\ref{closecase}]
We prove the claim for $i = 3, j = 2, k = 1$, the other cases follow from symmetry. Suppose contrary, for some $K$ and  $\lambda < 1/(44C_3 + 6C_4 + K), 1/(34440 C_1 C_3)$, we have $t_0$ such that $d(t_0 + e_1, t_0 + e_2) \leq K\lambda \mu, d(t_0, x_0) \leq 3 C_4\mu$ and $s\notcont{3}t_0$ whenever $s\in S_3(C_4, x_0)$, where $x_0$ is a point with $\rho(x_0) < 2\mu$.\\
Consider now points $y$ with $\rho(y) \leq 7C_3\mu, d(y + e_1, y  +e_2) \leq 7\lambda C_3\mu$, existence of such a point is granted by the Proposition~\ref{ex2prop}. Apply the Lemma~\ref{diaglemma} to $y$ and now we shall discard some of the diagrams by contractions with $t_0$.\\
Suppose that $y$ had the diagram A.1. By FNI $d(y, x_0) \leq (7C_3 + 2)\mu / (1-\lambda) \leq 8 C_3\mu$, and so $y, y + e_3 \in S_3(C_4, x_0)$. Hence $y, t_0$ and $y + e_3, t_0$ would be contracted by 1 or 2. However, from this we see that if $y + e_3\cont{1} t_0$ we get $\rho(y + e_1) = d(y + e_1, y+e_3 + e_1) \leq d(y  + e_1, t_0 + e_1) + d(t_0 + e_1, y + e_3 + e_1) \leq d(y + e_1, y + e_2) +  \lambda d(t_0, y) + d(t_0 + e_1, t_0 + e_2) + \lambda (d(t_0,y) + d(y, y + e_3)) \leq 7\lambda C_3 \mu + 3\lambda C_4\mu + \lambda d(x_0, y) +  \lambda K\mu + 3\lambda C_4\mu + \lambda d(x_0, y) + 7\lambda C_3\mu \leq \lambda(30 C_3 + 6C_4 + K)\mu < \mu$. On the other hand if $y + e_3\cont{2}t_0$, we get $\rho(y + e_2) \leq d(y + e_3 + e_2, y + e_2) + d(y + e_2 + e_3, y + e_2 + e_2) \leq d(y + e_3 + e_2, t_0 + e_2) + d(t_0 + e_2, y + e_2) + 14\lambda C_3\mu \leq \lambda (d(t_0,x_0) + d(x_0, y) + d(y, y + e_3)) + d(y + e_1, y + e_2) + d(t_0 + e_1, t_0 + e_2) + \lambda d(y, t_0) + 14 \lambda C_3\leq 3\lambda C_4\mu + 8\lambda C_3 \mu + 7\lambda C_3\mu + 7\lambda C_3\mu + \lambda K\mu + 3\lambda C_4 \mu + 8\lambda C_3\mu + 14 \lambda C_3 \mu\leq \lambda (44C_3 + 6 C_4 + K)\mu < \mu$.\\Similarly, if it was A.2 instead of A.1, we would have $y, y + e_1 \in S_3(C_4, x_0)$ and so contracting these two points with $t_0$, would give $\rho(y + e_2) = d(y + e_2 + e_1, y+e_2) \leq d(y+e_1 + e_2, y + 2e_1) + \lambda d(y + e_1, t_0)  + d(y + e_1, y + e_2) + \lambda d(y, t_0) + \lambda K\mu \leq \lambda (14C_3 + 15 C_3 + 3C_4 + 8C_3 + 3C_4 + K)\mu < \mu$.\\
Now consider the diagrams C.2 and D.2. We have $y, y + e_3 \in S_3(C_4, x_0)$ so contracting these points with $t_0$ must be by 1 or 2, so we immediately get $\rho(y + e_1) \leq \lambda (44C_3 + 6C_4 + K)\mu< \mu$.\\
Therefore, we can only have diagrams B.1, C.1, D.1 and diagram symmetric to B.1, which we shall refer to as B.2. Suppose now that $y$ with $\rho(y) \leq 7C_3\mu, d(y + e_1, y + e_2)\leq 7\lambda C_3\mu$ had diagram C.1 or D.1. Also, assume $\rho(y + e_3) \leq 7C_3\mu , d(y + e_3 + e_1, y +e_3+ e_2) \leq \lambda C_3\mu$, thus $y + e_3$ itself has one of the mentioned diagrams. Suppose that it had diagram B.1 or B.2. Without loss of generality, it was B.1, the other case is symmetric to this one.\\
Suppose $y$ has diagram C.1. Then given any point $z$ with $d(z,y) \leq 2\rho(y)$, suppose $d(y + e_1, z+ e_1),d(y + e_2, z+e_2) > 5\lambda \rho(y)$. Then $z \cont{3} y$ and so $y + e_1 \cont{2} z$, $y + e_2 \cont{1} z$. However, we can apply the Proposition~\ref{tps} to $(y; y + e_2, y + e_1, y + e_3; y + 2e_3)$ with constant $42C_3$, to see that $\diam N(z) \leq 800\lambda C_3\mu$, so after contracting $y, z$ we obtain $\rho(z) < 12C_3\mu$ and applying the Proposition~\ref{neighdiamProp} gives the contradiction, provided $\lambda < 1/(32800 C_1 C_3)$. So whenever $d(z,y) \leq 2\rho(y)$, we must have $d(y + e_1, z+ e_1) \leq 5\lambda \rho(y)$ or $d(y + e_2, z+e_2) \leq 5\lambda \rho(y)$. But contract $z$ with $y + e_2$ in the former case and with $y + e_1$ in the later to see that for some choice of distinct $i,j\in[3]$ we must have $d(z + e_i, y + e_1), d(z + e_j, y + e_1) \leq 20\lambda \rho(y)$, so $d(z + e_i, y), d(z + e_j,y) \leq 2\rho(y)$, thus we can repeat these arguments to points $z + e_i, z + e_j$. Doing so, we obtain a 2-way set of diameter at most $280 \lambda C_3\mu$ by considering the distance from $y + e_1$, if the point $z$ is dropped out. But, by the Lemma~\ref{1contrLemma}, we get such a 2-way set in every 3-way set, which is a contradiction by the Proposition~\ref{every2Prop}, since $\lambda < 1 /(840 C_3)$. Similarly we argue if $y$ had diagram D.1.\\
We conclude that if $y$ is as described and has diagram C.1 or D.1, then $y+e_3$ also has diagram among these two. Now, start from a point $y_0$ with $d(y_0 + e_1, y_0 + e_2) \leq 3\lambda C_3\mu, \rho(y_0) \leq 3C_3\mu$ and diagram C.1 or D.1, provided such a point exists. Define the sequence $y_n = y_0 + n e_3$, for all $n \geq 0$, aim is to show that this is Cauchy. By induction on $n$ we show that $\rho(y_n) \leq 7C_3\mu, d(y_n+e_1, y_{n+1} + e_1), d(y_n+e_2, y_{n+1} + e_2) \leq (n+3) \lambda ^ {n+1} C_3 \mu, d(y_n + e_1, y_n + e_2) \leq 3 C_3 \lambda ^{n+1}\mu$ and $y_n$ has either diagram C.1 or diagram D.1, which is true for $n = 0$.\\
Suppose the claim holds for all $m$ not greater than some $n \geq 0$. By the Proposition~\ref{tps} applied to $(y_0; p_1, p_2, p_3; y_n)$ with constant $18C_3$ with suitable $\{p_1, p_2, p_3\} = N(y_0)$ we get $d(y_1, y_{n+1}) \leq 288 \lambda C_3\mu$, so we infer that $\rho(y_{n+1}) \leq d(y_{n+1}, y_{n+1} + e_1) + d(y_{n+1} + e_1, y_{n+1} + e_2) \leq d(y_{n+1}, y_1) + d(y_1, y_0 + e_2) + d(y_0 + e_2, y_1 + e_2) + d(y_1 + e_2, y_2 + e_2) + \dots + d(y_n + e_2, y_{n+1} + e_2) \leq 7C_3\mu$ and $y_n +e_1\cont{3}y_n + e_2$ so $d(y_{n+1}  +e_1, y_{n+1} + e_2) \leq 3\lambda^{n+2} C_3$ therefore, $y_{n+1}$ must itself have diagram C.1 or D.1. If $y_n$ and $y_{n+1}$ have the same diagram, then we can see that $y_n + e_1 \cont{3} y_{n+1} + e_2$ and $y_{n+1} + e_1 \cont{3} y_n + e_2$ which is sufficient to establish the claim, as we obtain $d(y_{n+1} + e_1, y_{n+2} + e_1) \leq d(y_{n+1} + e_1, y_{n+2} + e_2) + d(y_{n+2} + e_2, y_{n+2} + e_1) \leq \lambda (d(y_n + e_1, y_{n+1} + e_1) + d(y_{n+1} + e_1, y_{n+1} + e_2)) + \lambda d(y_{n+1} + e_1, y_{n+1} + e_2) \leq \lambda d(y_n + e_1, y_{n+1} + e_2) + 6\lambda ^ {n+3} C_3 \mu \leq (n+3)\lambda^{n+2}C_3\mu + \lambda ^{n+2} C_3 \mu \leq (n+4)\lambda ^ {n+2} C_3\mu$. Likewise, we get the bound on $d(y_{n+1} + e_2, y_{n+2} + e_2)$. If the diagrams are different, it must be the case that $y_n + e_1 \cont{3} y_{n+1} + e_1$ and $y_n + e_2 \cont{3} y_{n+1} + e_2$, once again proving the claim, this time this is immediate.\\
Hence, if $y$ is a point such that $\rho(y) \leq 3C_3\mu, d(y + e_1, y + e_2) \leq 3\lambda C_3\mu$, then it can only have diagram B.1 or B.2. In the light of this, pick $y_0$ with $\rho(y_0) \leq C_3 \mu, d(y_0 + e_1, y_0 + e_2) \leq \lambda C_3 \mu$, whose existence is provided by the Proposition~\ref{ex2prop}, so it has diagram B.1, without loss of generality. Set $y_1 = y_0 + e_1$ and so have $\diam \{y_1, y_1 + e_1, y_1 + e_2\} \leq 3\lambda \rho(y_0)$ for the diagram for $y_0$. Also, by the Proposition~\ref{tps} applied to $(y_0; y_0 + e_3, y_0 + e_2, y_0 + e_1; y_1)$ with constant $6C_3$ we get $\rho(y_1) \leq 3C_3\mu$, so $y_1$ has diagram B.1 or B.2. If it is B.1 define $y_2$ to be $y_1 + e_1$, otherwise $y_1 + e_2$. Similarly, proceed, and as long as $y_k$ is defined and has one of these diagrams, define $y_{k +1 } = y_k + e_1$ when $y_k$ has diagram B.1, and set $y_{k+1} = y_k + e_2$ if it has diagram B.2. We now claim by induction on $k$ that $y_k$ is defined, $\rho(y_k) \leq 3C_3\mu$ and $\diam \{y_k, y_k + e_1, y_k + e_2\} \leq 3 (3\lambda)^k C_3\mu$. This is clear for $k = 0$.\\
Suppose the claim holds for some $k\geq0$. Then we have that $y_k$ has B.1 or B.2 for its diagram, suppose it is the former, we argue in the same way for the other option. Firstly, $y_{k+1}$ is defined. Then, from contractions implied by the diagram B.1, we get $\diam\{y_{k+1}, y_{k+1} + e_1, y_{k+1} + e_2\} \leq 3\lambda \diam\{y_k, y_k + e_1, y_k+e_2\}$. Finally, as $d(y_0, y_k) \leq (\rho(y_0) + \rho(y_k))/(1-\lambda) < 5C_3\mu$, we may apply the Proposition~\ref{tps} to $(y_0; y_0 + e_3, y_0 + e_2, y_0 + e_1; y_{k+1})$ with constant $6C_3$ to obtain $\rho(y_{k+1}) = d(y_{k+1}, y_{k + 1} + e_3) \leq d(y_{k+1}, y_0) + d(y_0, y_0 + e_3) +d (y_0 + e_3, y_{k+1} + e_3) \leq d(y_{k+1}, y_k) + d(y_k, y_{k-1}) + \dots + d(y_1, y_0) + \rho(y_0) + 96\lambda C_3\mu \leq 9\lambda C_3 \mu / (1-3\lambda) + 2C_3\mu + 96\lambda C_3\mu \leq 3C_3\mu$, which proves the claim. \\
This brings us to the conclusion that $(y_k)_{k\geq0}$ is a 1-way Cauchy sequence, providing us with a contradiction.\end{proof}

\begin{proof}[Proof of the Proposition~\ref{farcase}]
During the course of our argument, we shall prove a few auxiliary Lemmata, the last one being the Lemma~\ref{e3dist}, allowing us to conclude the proof.  It suffices to prove the claim for $i = 3, j = 2, k = 1$. Suppose contrary, there is $t_0$ with $d(t_0 + e_1, t_0 + e_2) > 10\lambda C_5\mu$, $d(t_0, x_0) \leq 3C_5\mu$ and whenever $s \in C_{3}(C_5, x_0)$, we must have either $s\cont{1}t_0$ or $s\cont{2}t_0$.\\

\begin{figure}
\caption{Possible diagrams of points $p$ with $d(p + e_1, p + e_2) \leq\lambda C_{5,1}\mu, \rho(p) \leq C_{5,1}\mu$}
\includegraphics[scale = 0.35]{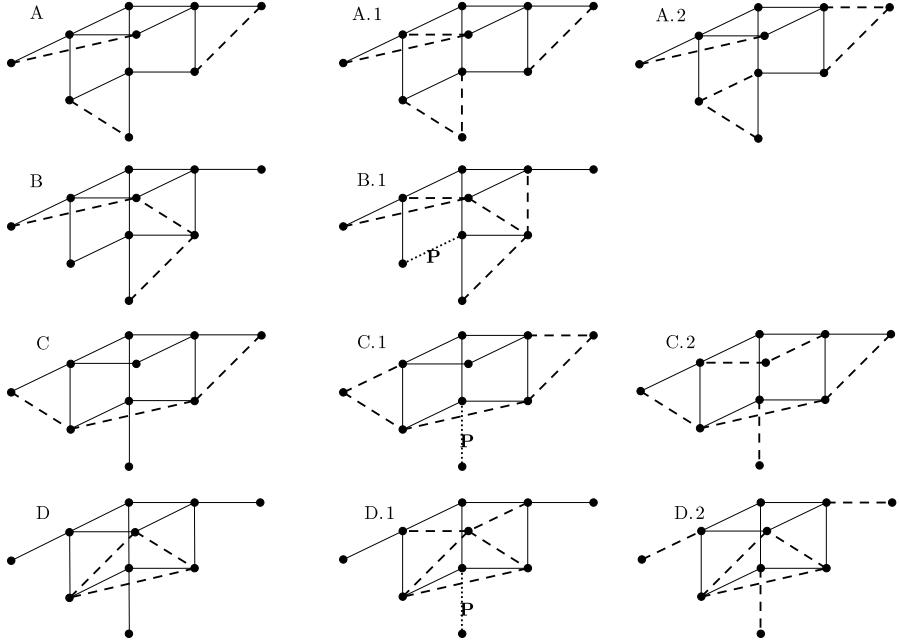}
\label{Figfar1}
\end{figure}

Set $C_{5,1} = 100C_3$ and consider the points $y$ with $\rho(y) \leq C_{5,1}\mu, d(y + e_1, y +e_2) \leq\lambda C_{5,1}\mu$. Note that such a point exists by the Proposition~\ref{ex2prop}. The possible diagrams of contractions are shown in the Figure~\ref{Figfar1}, and the arguments to justify these are provided in the Appendix. These are precisely the same diagrams as in the previous Proposition. Using $d(t_0 + e_1, t_0 + e_2) > 10\lambda C_5\mu$, we reject most of these.\\
\begin{itemize}
\item[B.1] Suppose that $y$ as above has diagram B.1. First of all, as $\lambda < 1/(4920 C_1 C_{5,1})$, apply the Proposition~\ref{tps} to $(y; y + e_3, y + e_2, y + e_1; y)$ with constant $6 C_{5,1}$ to see that in particular $y, y + e_1, y + e_2, y + e_3$ are all in $C_3(x_0, C_5)$, as $d(y, x_0) \leq (C_{5,1} + 2)\mu/(1-\lambda)$, $\rho(y) \leq C_{5,1}\mu$, $d(y, y + e_3) \leq C_{5,1}\mu, d(y + e_1, y + e_1 + e_3) \leq (2 + 96\lambda) C_{5,1}\mu, d(y + e_2, y + e_2 + e_3) \leq \lambda C_{5,1}\mu$ and $d(y+ e_3, y + 2e_3) \leq (2+3\lambda) C_{5,1}\mu$.\\
If $t_0 \cont{1} y$, then contract $t_0, y + e_3$ to get $\rho(y + e_1) < 6\lambda (C_{5,1} + C_5)\mu < \mu$ when $t_0 \cont{1} y + e_3$ or $d(t_0 + e_1, t_0 + e_2) \leq d(t_0 + e_1, y + e_1) + d(y + e_1, y + e_3 + e_2) + d(y + e_3 + e_2, t_0 + e_2) \leq \lambda (3C_{5,1} + 3C_5)\mu + 3\lambda C_{5,1}\mu + (3C_{5,1} + 3C_5)\mu < 10\lambda C_5\mu$ otherwise, both of which are not permissible.\\
\item[C.1] Suppose that $y$ as above has diagram C.1. Then $d(y, y + e_3) \leq C_{5,1}\mu, d(y + e_1, y + e_1 + e_3), d(y + e_2, y +e_2 + e_3) \leq 3\lambda C_{5,1}\mu, d(y + e_3, y + 2e_3) \leq 96\lambda C_{5,1}\mu$. Also $d(y, x_0) \leq (C_{5,1} + 2)\mu/(1-\lambda), \rho(y) \leq C_{5,1}\mu$, so $y, y + e_1, y + e_2, y + e_3 \in S_3(x_0, C_5)$. Without loss of generality $y \cont{1} t_0$. But if $y + e_2 \cont{1} t_0$, then $\rho(y + e_1) = d(y + e_1, y + e_1 + e_2) \leq \lambda d(y, t_0) + \lambda d(y + e_2, t_0) \leq 6\lambda (C_{5,1} + C_5)\mu < \mu$. However, $y + e_2 \cont{2} t_0$ is impossible as well, for that implies $d(t_0 + e_1, t_0 + e_2) \leq d(t_0 + e_1, y + e_1) + d(y + e_1, y + 2e_2) + d(y+2e_2, t_0 + e_2) \leq 6\lambda (C_{5,1} + C_5) \mu + 7\lambda C_{5,1} < 10 \lambda C_5\mu$.
\item[C.2] Assume that $y$ as above has diagram C.2. First of all apply the Proposition~\ref{n1neighProp} to $y$ (we have $\lambda < 1/(78 C_{5,1})$) to see that $d(y + e_1, y + e_1 + e_3), d(y + e_2, y + e_2 + e_3) \leq 9C_{5,1}\mu$. Also $d(y + e_3, y + 2e_3) \leq \lambda C_{5,1}\mu, \rho(y)\leq C_{5,1}\mu, d(y, x_0) \leq (C_{5,1} + 2)\mu/(1-\lambda)$, so $y, y + e_1, y + e_2, y + e_3 \in S_3(x_0, C_5)$. Without loss of generality $y \cont{1} t_0$. If $y + e_1 \cont{1} t_0$ then $\rho(y + e_1) \leq d(y + e_1, y + 2e_1) + d(y + 2e_1, y + e_1 + e_2) \leq \lambda (d(y, t_0) + d(y + e_1, t_0)) + 2\lambda C_{5,1}\mu \leq 6\lambda (C_{5,1} + C_5)\mu + 2\lambda C_{5,1} < \mu$. So, we must have $y + e_1 \cont{2} t_0$, but this also implies contradiction as $d(t_0 + e_1, t_0 + e_2) \leq d(t_0 + e_1, y + e_1) +d(y + e_1, y + y_1 + e_2) + d(y + e_1 + e_2, t_0 + e_2) \leq \lambda d(y, t_0) + \lambda C_{5,1}\mu + \lambda d(y + e_1, t_0) \leq 6\lambda (C_5 + C_{5,1})\mu + \lambda C_{5,1}\mu < 10\lambda C_5 \mu$.
\item[D.1] Let $y$ have diagram D.1. Then $d(y, y + e_3) \leq C_{5,1}\mu, d(y + e_1, y + e_1 + e_3), d(y + e_2, y +e_2 + e_3) \leq 3\lambda C_{5,1}\mu, d(y + e_3, y + 2e_3) \leq 96\lambda C_{5,1}\mu$. Also $d(y, x_0) \leq (C_{5,1} + 2)\mu/(1-\lambda), \rho(y) \leq C_{5,1}\mu$, so $y, y + e_1, y + e_2, y + e_3 \in S_3(x_0, C_5)$. Without loss of generality $y \cont{1} t_0$. If $t_0 \cont{1} y + e_1$, then $\rho(y + e_1) = d(y + e_1, y + 2e_1) \leq d(y + e_1, t_0 + e_1) + d(t_0 + e_1, y +2e_1) \leq \lambda (d(y, t_0) + d(t_0, y + e_1)) \leq 6\lambda (C_5 + C_{5,1}) < \mu$. On the other hand $t_0 \cont{2} y + e_1$ implies $d(t_0 + e_1, t_0 + e_2) \leq d(t_0 + e_1, y + e_1) + d(y + e_1, y + e_1 + e_2) + d(y + e_1 + e_2, t_0 + e_2) \leq \lambda (6C_5 + 7C_{5,1})\mu < 10 \lambda C_5\mu$. Thus, $y$ cannot have diagram D.1.
\item[D.2] Suppose that $y$ as above has diagram D.2. First of all apply the Proposition~\ref{n1neighProp} to $y$ to see that $d(y + e_1, y + e_1 + e_3), d(y + e_2, y + e_2 + e_3) \leq 9C_{5,1}\mu$. Also $d(y + e_3, y + 2e_3) \leq \lambda C_{5,1}\mu, \rho(y)\leq C_{5,1}\mu, d(y, x_0) \leq (C_{5,1} + 2)\mu/(1-\lambda)$, so $y, y + e_1, y + e_2, y + e_3 \in S_3(x_0, C_5)$. Without loss of generality $y \cont{1} t_0$. Now contract $y + e_2,t_0$. If these are contracted by 1, then $\rho(y + e_1) \leq d(y+e_1, y + e_1 +e_2) + d(y + e_1 + e_2, y + e_1 + e_3) \leq \lambda (6C_5 + 8C_{5,1})\mu < \mu$, which is a contradiction. Therefore $t_0\cont{2} y  +e_2$, which gives $d(t_0 + e_1, t_0 + e_2) \leq d(t_0 + e_1, y + e_1) + d(y + e_1, y + 2e_2) + d(y + 2e_2, t_0 + e_2) \leq \lambda (6C_5 + 8C_{5,1})\mu < 10\lambda C_5\mu$.
\end{itemize}
Thus, we are only left with diagrams A.1 and A.2. Let A.1' and A.2' be symmetric to these after swapping $e_1$ and $e_2$. Let $y$ once again be the same point as before. We know distinguish the possibilities for contractions with $t_0$.
\begin{itemize}
\item If $y$ has diagram A.1, then $y, y + e_1, y + e_3 \in S_3(x_0, C_5)$, and it is easy to see that $t_0,y$ and $t_0, y+e_1$ are contracted in the same direction, while $t_0, y + e_3$ is contracted in the other. Similarly, we see the possible contractions with $t_0$ for diagram A.1'.
\item If $y$ has the diagram A.2, we have all the points in $\{y\}\cup N(y)$ being members of $S_3(x_0, C_5)$, and pairs $t_0, y$ and $t_0, y + e_1$ must be contracted in the different directions (otherwise $\rho(y + e_2) < \mu$). Same holds for the pairs $y + e_1, t_0$ and $y + e_3, t_0$. From this we see that $t_0 \cont{2} y, t_0 \cont{2} y + e_3, t_0 \cont{1} y + e_1$. Analogously, we classify the contractions for A.2'. 
\end{itemize}

\begin{lemma}\label{A1seq} Let $K \leq C_{5,1}$. There is no sequence $(y_k)_{k\in I}$ for suitable index set $I \subset \mathbb{N}_0$, with the following properties:\begin{enumerate}
\item $y_0$ is defined, has $\rho(y_0) \leq K / (2 + 6\lambda), d(y_0 + e_1, y_0 + e_2) \leq \lambda K/(2+\lambda)\mu$,
\item If $y_k$ is defined, and satisfies $\rho(y_k) \leq K\mu, d(y_0 + e_1, y_0 + e_2) \leq \lambda K\mu$, then $y_k$ has diagram A.1 or A.1', and we define $y_{k+1} = y_k + e_i$, with $i= 1$ when diagram of $y_k$ is A.1 and $i = 2$ otherwise.
\end{enumerate}
\end{lemma}

\begin{proof}[Proof of the Lemma~\ref{A1seq}]
By induction on $k$, we claim that $y_k$ is defined and $\diam \{y_k, y_k + e_1, y_k + e_2\} \leq (3\lambda)^k K\mu / (2 + 6\lambda)$. This trivially holds for $k=0$. Also, without loss of generality $y_0$ has diagram A.1.\\
Suppose that the claim holds for all $k'$ not greater than $k$, where $k \geq 0$. Observe that $d(y_0, y_k) \leq d(y_0, y_1) + d(y_1, y_2) + \dots d(y_{k-1}, y_k) \leq (1 + 3\lambda + \dots + (3\lambda)^{k-1})K\mu/(2 + 6\lambda) <  \frac{1}{(1-3\lambda)(2 + 6\lambda)} K\mu$. Now, contract $y_0 + e_3, y_k$. It is contracted neither by 1 nor by 2, since we either get $\rho(y_0 + e_1) < \mu$ or $\rho(y_0 + e_2) < \mu$. Hence $y_k\cont{3}y_0 + e_3$, so $d(y_k + e_3, y_0 + e_3) \leq d(y_k + e_3, y_0 + 2e_3) + d(y_0 + 2e_3, y_0 + e_3) \leq \frac{\lambda(3 - 6\lambda)}{(1- 3\lambda)(2 + 6\lambda)} K\mu < 2\lambda K \mu$. Finally, we establish $\rho(y_k) \leq (2 + 6\lambda) K\mu/(2 + 6\lambda) = K\mu$, which combined with $d(y_k + e_1, y_k + e_2) \leq \lambda K\mu$ gives that $y_k$ itself has diagram A.1 or A.1'. Hence $y_{k+1}$ is defined, and $\diam\{y_{k+1}, y_{k+1}+e_1, y_{k+1}+e_2\} \leq 3\lambda \diam\{y_k, y_k + e_1, y_k + e_2\}$, as desired.\\
However, this shows that $(y_k)_{k\geq0}$ is a 1-way Cauchy sequence, which is not allowed. Therefore, we reach the contradiction, and the end of the proof. \end{proof}

\begin{corollary} There exists a point $y$ with $\rho(y) \leq 3C_3\mu, d(y + e_1, y + e_2) \leq 3\lambda C_3\mu$ with diagram A.2 or A.2'.
\label{farA2}\end{corollary}

\begin{proof}[Proof of the Corollary~\ref{farA2}] Suppose contrary, and let $y_0$ be a point with $\rho(y_0) \leq C_3\mu, d(y_0 + e_1, y_0 + e_2) \leq \lambda C_3\mu$, given by the Proposition~\ref{ex2prop}. We shall now define a sequence $(y_k)$ inductively, as long as we can. The starting point $y_0$ is as above. Given $y_k$, provided it satisfies $\rho(y_k) \leq 3C_3 \mu, d(y_k +e_1, y_k + e_2) \leq 3\lambda C_3\mu$, define $y_{k+1}$ to be $y_k + e_1$ when $y_k$ has diagram A.1 and $y_k + e_2$ if $y_k$ has diagram A.1', (note that by assumption these two are the only permissible diagrams). But this gives contradiction by the Lemma~\ref{A1seq} with $K = 3C_3$.\end{proof}

\begin{corollary} We have $d(t_0 + e_1, t_0 + e_2) > 5C_3\mu$. \label{t0far}\end{corollary}
\begin{proof}[Proof of the Corollary~\ref{t0far}]
Suppose contrary.  In order to reach the contradiction, we shall obtain a Cauchy sequence as in the previous proof. Consider a point $y$ with $\rho(y) \leq 36C_3\mu, d(y +e_1, y + e_2) \leq 36\lambda C_3\mu$. Assume that this point has diagram A.2. Recall that we have $t_0 \cont{1} y + e_1, t_0 \cont{2} y$. This gives $\rho(y + e_1) \leq C_{5,1}\mu$ and from contractions of $\{y\} \cup N(y)$ we get that $d(y + e_1 + e_1, y + e_1 + e_2) \leq \lambda C_{5,1}\mu$ holds as well, so $y + e_1$ has one of the four diagrams considered so far. However, we immediately see that it is not possible for $y + e_1$ to have diagram A.2, for $t_0 \cont{1} y + e_1$.\\
Suppose that $y + e_1$ had diagram A.2'. Firstly, suppose that $y + e_1\cont{3} y + e_2 + e_3$. Then contract $y, y + 2e_3$. If it is by 3, we have $\rho(y + 2e_3) < \mu$, otherwise we obtain $\rho(y +e_3) < \mu$. Hence $y + e_1 \cont{2} y + e_2 + e_3$. This further implies $y+e_1\cont{2}y +2e_2$ (or otherwise $\rho(y + 2e_1) < \mu$). However $y +2e_2\in S_3(x_0, C_5)$, so contract $y +e_2, t_0$ to get a contradiction.\\
Suppose now that $y$ has diagram A.1 and $\rho(y) \leq 17C_3\mu, d(y + e_1, y + e_2) \leq 17\lambda C_3\mu$. If $y + e_1$ has diagram A.2, then $y + e_1 \cont{2} t_0, y + e_1 + e_3 \cont{2} t_0, y +2e_1 \cont{1} t_0$. But $y$ has diagram A.1, so $t_0$ contracts with $y + e_1, y$ in the same direction, thus in $e_2$, and $t_0, y+e_3$ in the other, i.e. $e_1$. However, then $\diam N_1(x+2e_1) < 10 \lambda C_5\mu$, which is contradiction with Proposition~\ref{neighdiamProp} used with constant $10C_5$ after contracting $y,y +2e_1$.\\
Assume that $y + e_1$ has diagram A.2'. Thus $t_0 \cont{1} y + e_1, t_0 \cont{1} y + e_1 + e_3$ and $t_0 \cont{2} y + e_1 + e_2$. As $y$ has diagram A.1, we have $t_0 \cont{1} y$ and $t_0 \cont{2} y + e_3$. But, as $d(t_0 + e_1, t_0 + e_2) \leq 5C_3\mu$, we have $y + e_1 + e_2$ $100C_3-$good, so by the previous discussion $y + e_1 + e_2$ can only have diagram A.1 or A.1' (as $y + e_1$ is $36C_3-$good).\\
If $y + e_1 + e_2$ has diagram A.1 then $t_0 \cont{1} y + e_1 + e_2 + e_3$, so $\rho(y + e_1) < \mu$, so we may assume $y + e_1 + e_2$ has diagram A.1', which implies $t_0 \cont{1} y + e_1 + e_2 +e_3$. Look at pairs $y + 2e_1, y + 2e_1 + e_3$ and $y + 2e_1 + e_2, y + 2e_1 + e_3$, both have length at most $6C_3\mu$, so cannot be contracted by 2, as otherwise $d(t_0 + e_1, t_0 + e_2) < 10 C_5\mu$. Suppose that at least one of these pairs is contracted by 1. Then apply the Proposition~\ref{tps} to $(y;y + 3e_1, y + 2e_1, y+ e_1; y + e_2)$ with constant $10C_5$ (since $\lambda < 1/(8200C_1C_5)$), to see that $\rho(y +3e_3) < \mu$. Hence, the two considered pairs are contracted by 3. But, contract $y + e_2, y + 2e_1 + e_3$ to get $\rho(y + 2e_1 + e_3) < \mu$ or $d(t_0 + e_1, t_0 + e_2) < 200 \lambda C_5\mu$ giving $\rho(y + e_2) < \mu$.\\
Now, start from $y_0'$ with $\rho(y_0') \leq C_3\mu, d(y_0' + e_1, y_0' + e_2) \leq \lambda C_3\mu$, given by the Proposition~\ref{ex2prop}. If $y_0'$ has diagram A.1 or A.1' set $y_0 = y_0'$, otherwise set $y_0 = y_0' + e_1$ if the diagram is A.2, and $y_0 = y_0' + e_2$ if the diagram is A.2'. Hence, $y_0$ satisfies $\rho(y_0) \leq 6C_3\mu, d(y_0 + e_1, y_0 + e_2) \leq 6\lambda C_3\mu$, and defining the sequence as in the Lemma~\ref{A1seq}, gives a contradiction for $K = 17C_3$ by the discussion above.\end{proof}

\begin{lemma}\label{e3dist} Suppose that $y_1, y_2$ are two points with $\rho(y_1), \rho(y_2) \leq C_3\mu$. Then $d(y_1 + e_3, y_2 + e_3) \leq 40\lambda C_3\mu$.\end{lemma} 
\begin{proof}[Proof of the Lemma~\ref{e3dist}]
Recall that we have a point $y_0$ with $\rho(y_0) \leq 6C_3\mu, d(y_0 + e_1, y_0 + e_2) \leq 6\lambda C_3\mu$, with diagram A.2 or A.2', given by the Corollary~\ref{farA2}. Without loss of generality it is A.2.\\
Let $z$ be any point with $\rho(z) \leq C_3\mu$. We shall prove $d(y_0 + e_3, z + e_3) \leq 20\lambda C_3\mu$, which is clearly sufficient. Note that we have $d(z, x_0) \leq (C_3 + 2)\mu/(1-\lambda) \leq C_5\mu$,$d(z, t_0) \leq d(z,x_0) + d(x_0, t_0)\leq (C_3 + 2)\mu/(1-\lambda) + 3C_5\mu \leq 4C_5\mu$, and similarly $d(y_0, z) \leq 4C_5\mu$ and $y_0,z \in S_3(x_0, C_5)$.\\
Assume $t_0 \cont{1} z$. Recall that $t_0\cont{2}y_0$. If $y_0 \cont{1} z$, then $d(t_0 + e_1, t_0 + e_2) \leq d(t_0 + e_1, z + e_1) + d(z + e_1, y_0 + e_1) + d(y_0 + e_1, y_0 + e_2) + d(y_0 + e_2, t_0 + e_2) \leq \lambda 4C_5\mu + \lambda 7C_3/(1-\lambda) + 6\lambda C_3\mu + 4\lambda C_5\mu < 5C_3\mu < d(t_0 +e_1, t_0 + e_2)$, which is a contradiction. Similarly we discard the case $y_0 \cont{2} z$, as then $d(t_0 + e_1, t_0 + e_2) \leq d(t_0 + e_1, z + e_1) + d(z + e_1, z+e_2) + d(z + e_2, y_0 + e_2) + d(y_0 + e_2, t_0 + e_2) \leq 5C_3\mu$. Therefore, $y_0 \cont{3} z$, so $d(y_0 + e_3, z + e_3) \leq \lambda 7C_3\mu /(1-\lambda) < 8\lambda C_3\mu$.\\
Thus, we must have $z \cont{2} t_0$. But we cannot have neither $y_0 + e_1 \cont{1} z$ nor $y_0 + e_1 \cont{2} z$, for otherwise we obtain $d(t_0 + e_1, t_0 + e_2) \leq d(t_0 + e_1, y_0 + 2e_1) + d(y_0 + 2e_1, z + e_2) + d(z + e_2, t_0 + e_2) \leq \lambda d(t_0, y_0 + e_1) + d(y_0 + 2e_1, y_0 + e_1 + e_2) + \lambda d(y+e_1,z) + 2\rho(z) + \lambda d(z, t_0) \leq 5C_3\mu$. Hence, we get $y_0 + e_1 \cont{3} z$, so $d(y_0 + e_3, z + e_3) \leq \lambda d(y_0 + e_1, z) + d(y_0 + e_1 + e_3, y_0 + e_3) \leq 14\lambda C_3\mu + 6\lambda C_3\mu = 20\lambda C_3\mu$, as desired.\end{proof}

We are now ready to establish the final contradiction. By the Proposition~\ref{ex2prop}, we have points $x_1, x_2, x_3$ with whenever $\{i,j,k\} = [3]$, we have $\rho(x_i) \leq C_3\mu, d(x_i + e_j, x_i + e_k) \leq \lambda C_3\mu$. First of all, $x_1, x_2, x_3$ all belong to $S_3(x_0, C_5)$, since $d(x_0, x_i) \leq (C_3 + 2)\mu / (1-\lambda)$. Suppose that for some $i,j$ we have $t_0\cont{1}x_i$ and $t_0 \cont{2}x_j$. Then, by the triangle inequality and FNI, $d(t_0 + e_1, t_0 + e_2) \leq d(t_0 + e_1, x_i + e_1) + d(x_i + e_1,x_i) + d(x_i, x_j) + d(x_j, x_j + e_2) + d(x_j + e_2, t_0 + e_2) \leq \lambda (d(t_0,x_0) + d(x_0, x_i)) + \rho(x_i) + (\rho(x_i) + \rho(x_j))/(1-\lambda) + \rho(x_j) + \lambda (d(x_j,x_0) + d(x_0, t_0)) \leq \lambda (3C_5\mu + (\rho(x_0) + \rho(x_i))/(1-\lambda)) + C_3\mu + 2C_3\mu/(1-\lambda) + C_3\mu + \lambda((\rho(x_j) + \rho(x_0))/(1-\lambda) + 3C_5\mu) \leq 5C_3\mu$, which is not possible, hence $t_0$ contracts with $x_1, x_2, x_3$ in the same direction, $e_1$ without loss of generality. But also the Lemma~\ref{e3dist} gives $\diam\{x_1 + e_3, x_2 + e_3, x_3 + e_3\} \leq 40\lambda C_3\mu$, and $\diam\{x_1 + e_1, x_2 + e_1, x_3 + e_1\}\leq 8\lambda C_5\mu$ so $\diam N(x_1) \leq 9\lambda C_5\mu$, which is a contradiction due to the Proposition~\ref{neighdiamProp}. \end{proof}

Now combine the Corollary~\ref{casescor} with the Propositions~\ref{closecase} and~\ref{farcase} to obtain a contradiction.\end{proof}

\section{Concluding Remarks}
Let us restrict our attention once again to the Austin's conjecture in its generality. Thus, we can formulate the following hypothesis, which serves to capture its essence.
\begin{conjecture}\label{conjA} Let $n$ be a positive integer and $\lambda$ a real with $0 \leq \lambda < 1$. Suppose $(\mathbb{N}_0^n,d)$ be an $n$-dimensional $\lambda$-contractive grid, that is a pseudometric space with the property that given $x,y \in \mathbb{N}_0^n$ we have some $i \in [n]$ with $d(x + e_i, y + e_i) \leq \lambda d(x,y)$. Then there is a 1-way Cauchy sequence. 
\end{conjecture}
Of course, having in mind the proof of~\ref{thmMain} there are similar versions of this hypothesis that can be asked. Recall $\mu = \inf \rho(x)$, where $x$ ranges over all points in the grid and set $\mu_\infty = \lim_{k\to\infty} \inf_{x \in S_k} \rho(x)$, where $S_k$ is the $n-$way set generated by $(k,k,\dots,k)$. Also, say that a pseudometric space is a contractive grid if it is $n-$dimensional $\lambda-$contractive grid, for $0\leq \lambda < 1$ and a positive integer $n$.
\begin{question} Can $\mu > 0$ occur in a contractive grid? \end{question}
\begin{question} Can $\mu_\infty = \infty$ occur in a contractive grid? \end{question}  
Note that the proof we give here are of combinatorial nature, with the flavor of Ramsey theory, and seems that the complete proof of~\ref{conjA} should be based on a similar approach. Note that the proof of Generalized Banach Theorem rests on Ramsey Theorem. 
\begin{question} What is the combinatorial principle behind the Conjecture~\ref{conjA}? Is it related to Ramsey theory? \end{question}
Finally, we consider some points of the proof Theorem~\ref{thmMain} and pose the following questions and conjectures.
\begin{question} What is the relation between $k-$way sets in higher dimensional grids?\end{question}
\begin{conjecture} For each $k \geq 2$ there is a positive constant $C_k$ with the following property:\\
Given a $k-$colouring of a countably infinite graph $G$, we can find sets of vertices $A_1, A_2, \dots, A_{k-1}$ which cover the graph, and for suitable colors $c_1, c_2, \dots, c_{k-1} \in [k]$, we have $\diam_{c_i} G[A_i] \leq C_k$ for all $i\in[k-1]$. \end{conjecture} 
\begin{question} What conditions on a colouring $c$ of $\mathbb{N}_0^n$ ensure that the colouring is essentially trivial, that is, it is monochromatic on a $n-$way set?\end{question}
\textbf{Acknowledgments.} Let me thank Trinity College for making this project possible through its funding and support. I am also endebted to the anonymous referee for the very careful reading of the paper, improving the presentation significantly.
\appendix
\section{Discussion of The Possible Contraction Diagrams}
\begin{figure}
\includegraphics[scale = .3]{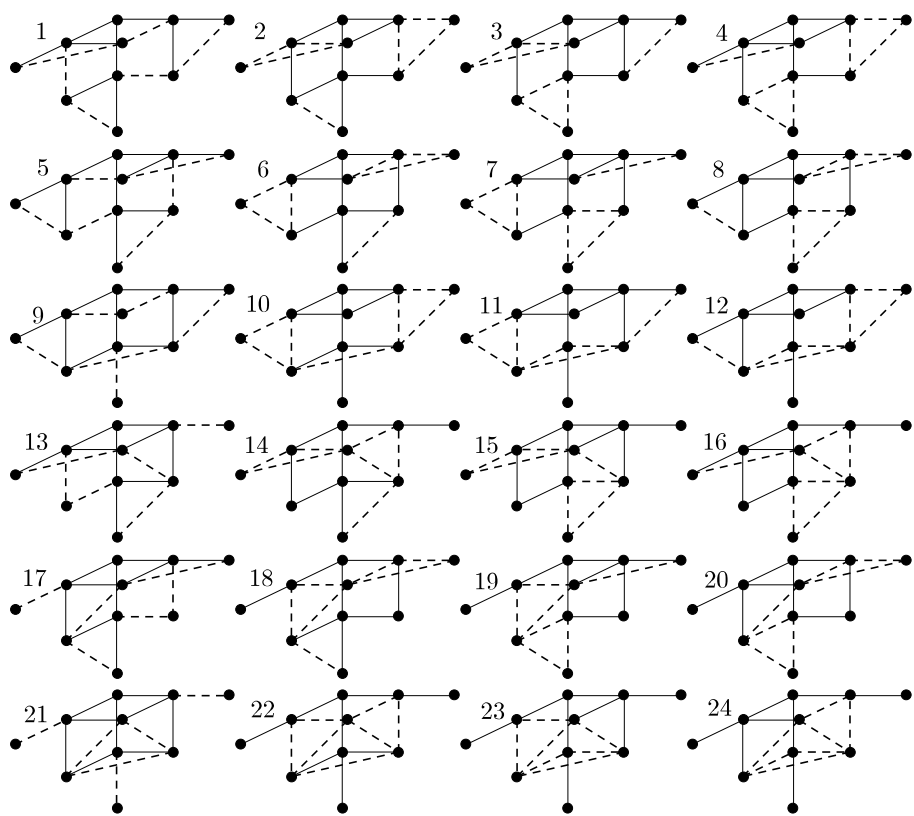}
\label{alldiagrams}
\caption{All possible contraction diagrams}
\end{figure}
In this appendix we discuss how we obtain the possible diagrams for contractions in the latter part of the proof of the Proposition~\ref{keyprop}. For the this discussion we assume the proposition prior to the Proposition~\ref{ex1prop} to hold.\\
Let us start with a point $x$ with $\rho(x) \leq K\mu$, for some $K \geq 1$. Consider first the contractions of the long edges, that is those of the form $x + e_i, x + e_j$, where $i,j$ are distinct elements of $[3]$. If two such edges are contracted in the same direction, say $k$, then $\diam N(x + e_k) \leq 4\lambda K \mu$. Furthermore, we can contract $x, x + e_k$, to get $\rho(x_k) \leq (2 + 5\lambda) K \mu$, which is a contradiction due to the Proposition~\ref{neighdiamProp}, provided $\lambda < 1/(164 C_1 K)$, which we shall assume is the case. Thus, all three long edges must be contracted in the different directions.\\
Contract now the short edges, i.e. those edges of the form $x, x + e_i$, for some $i\in[3]$. Given such an edge, there is a unique long edge $x + e_j, x + e_k$, such that $\{i,j,k\} = [3]$. We say that these edges are orthogonal. Suppose that a short edge $x + e_i$ is not  contracted in the same direction as its orthogonal long edge. Then $x + e_i$ must be contracted in the same direction $e_l$ as $x + e_i, x + e_j$, for some $j\not= i$. Let $k$ be such that $\{i, j, k\} = [3]$. Then $x + e_k$ cannot be contracted in the same direction as $x + e_i$, as otherwise $\rho(x + e_l) \leq 3\lambda K \mu < \mu$, which is impossible. So, $x + e_k$ is contracted in the same direction as one of its nonorthogonal long edges. Hence $\diam \{x + e_l, x + e_l +e_i, x + e_l + e_j\}, \diam \{x + e_m, x + e_m + e_k, x + e_m + e_n\} \leq 3\lambda \rho(x)$ holds for some $m,n \in[3]$ where $m \not = l$ and $n \not = k$. From this we can conclude that contractions in $\{x\} \cup N(x)$ can only give the diagrams shown in the Figure~\ref{alldiagrams}. There, an edge shown as dash line implies that its length is at most $3\lambda \rho(x)$.\\
\subsection{Diagrams in the proof of the Proposition~\ref{ex2prop}}
As in the proof of the Proposition~\ref{ex2prop} we consider point $y$ with $d(y + e_1, y + e_3) \leq \lambda C_{3,1} \mu$ and $\rho(y) \leq C_{3,1}\mu$, that is we set previously considered $K$ to be $C_{3,1}$ instead, and so assume $\lambda < 1/(164 C_1 C_{3,1})$. Consider the possible diagrams of contractions of edges in $\{y\} \cup N(y)$. Recall that our assumption is that there is no point $x$ with $\rho(x) \leq C_3 \mu$ and $d(x  +e_1, x + e_2) \leq \lambda C_3 \mu$. We now describe how to reject all diagrams except 2,4,6,11,15,23.
\begin{itemize}
\item[1] Immediately we get $\rho(y + e_3) \leq 4\lambda C_{3,1} \mu < \mu$.
\item[3] We have $\rho(y + e_1) \leq 4\lambda C_{3,1} \mu < \mu$.
\item[5] Similarly to previous ones $\rho(y + e_1) \leq 7\lambda C_{3,1}\mu$.
\item[7] We get $\rho(y + e_3) \leq 4\lambda C_{3,1} \mu < \mu$.
\item[8] Have $\rho(y + e_2) \leq (2 + 3\lambda)C_{3,1}\mu, d(y+e_2 + e_1, y + e_2 + e_2) \leq 3\lambda C_{3,1}$, but we assume that there are no such points.
\item[9] Diameter of $N(y)$ is at most $7\lambda C_{3,1}\mu$ and $\rho(y) \leq C_{3,1}\mu$ so apply the Proposition~\ref{neighdiamProp}, provided $\lambda < 1/(287 C_1 C_{3,1})$.
\item[10] Diameter of $N(y)$ is at most $10\lambda C_{3,1}\mu$ and $\rho(y) \leq C_{3,1}\mu$ so apply the Proposition~\ref{neighdiamProp}, provided $\lambda < 1/(410 C_1 C_{3,1})$.
\item[12] We apply the Proposition~\ref{tps} to $(y; y + e_2, y + e_1, y + e_3; y)$ with constant $9C_{3,1}$, so $\rho(y + e_2) \leq 144\lambda C_{3,1}\mu < \mu$, as long as $\lambda < 1/ (7380 C_1 C_{3,1})$.
\item[13] Use the Proposition~\ref{n1neighProp} to get $\rho(y + e_1) \leq (11 + 9\lambda)C_{3,1}\mu$ and $d(y + e_1 + e_1, y + e_1 + e_2) \leq 3\lambda C_{3,1}\mu$, as $\lambda < 1/(936 C_{3,1})$. This is a contradiction as $C_3 > 12 C_{3,1}$.
\item[14] Apply the Proposition~\ref{tps} to $(y; y +e_3, y + e_2, y+e_1;y)$ with constant $9C_{3,1}$ to get $\rho(y + e_2) \leq 144\lambda C_{3,1}\mu$. Here we need $\lambda < 1/(7380 C_1 C_{3,1})$.
\item[16] As 14.
\item[17] As for 13, get $\rho(y + e_2) \leq (11 + 9\lambda)C_{3,1}\mu$ and $d(y + e_2 + e_1, y + e_2 + e_2) \leq 3\lambda C_{3,1}\mu$.
\item[18] Apply the Proposition~\ref{tps} to $(y; y +e_1, y + e_3, y+e_2;y)$ with constant $9C_{3,1}$ to get $\rho(y + e_2) \leq 144\lambda C_{3,1}\mu < \mu$.
\item[19] Apply the Proposition~\ref{tps} to $(y; y +e_1, y + e_3, y+e_2;y+e_3)$ with constant $9C_{3,1}$ to get $\rho(y + e_2) \leq (2 + 6\lambda)C_{3,1}\mu, d(y + e_2 + e_1, y + e_2 + e_2) \leq 3\lambda C_{3,1}\mu$.
\item[20] As 18.
\item[21] Use the Proposition~\ref{n1neighProp} to get $\rho(y + e_3) \leq (9 + 3\lambda)C_{3,1}\mu$ and $d(y + e_3 + e_1, y + e_3 + e_2) \leq 3\lambda C_{3,1}\mu$, as $\lambda < 1/(78 C_{3,1})$.
\item[22] Have $\diam N(y) \leq 7\lambda K\mu$ which is in contradiction with the Proposition~\ref{neighdiamProp}, when $\lambda < 1/(287 C_1 C_{3,1})$.
\item[24] Apply the Proposition~\ref{tps} to $(y; y + e_1, y + e_2, y + e_3; y + e_2)$ with constant $6 C_{3,1}$ to get $\rho(y +e_2)\leq 96\lambda C_{3,1}\mu$.
\end{itemize}

Therefore, we have that for the $y$ given above, provided $\lambda <1/ (7380 C_1 C_{3,1})$, we can only have diagrams 2, 4,6,11,15,23. However, in all of these diagrams we can classify contractions more precisely.
\begin{itemize}
\item[2] Observe that we cannot have $y + e_1 \cont{2} y$ or $y + e_1 \cont{3} y$ as the first one of these gives $\rho(y + e_2) \leq 10\lambda C_{3,1}\mu < \mu$, while the later implies $\rho(y + e_1) \leq 10 C_{3,1}\mu < \mu$. Hence $y + e_1 \cont{1} y$. Similarly, we must have $y \cont{2} y + e_3$, otherwise we get a point $p$ with $\rho(p) \leq 10\lambda C_{3,1}\mu < \mu$.
\item[4] As in 2, if we do not have $y \cont{3} y + e_1$ and $y \cont{2} y + e_2$, we obtain a point $p$ with $\rho(p) \leq 10 \lambda C_{3,1}\mu < \mu$.
\item[6] As in 2, if we do not have $y \cont{2} y + e_2$ and $y \cont{1} y + e_3$, we obtain a point $p$ with $\rho(p) \leq 10 \lambda C_{3,1}\mu < \mu$.
\item[11] If $y \cont{3} y + e_3$, then $\rho(y + e_3) \leq 10\lambda C_{3,1}\mu <\mu$. On the other hand, if $y \cont{2} y + e_3$, then $\diam N(y) \leq 8\lambda C_{3,1} \mu$ and $\rho(y) \leq C_{3,1} \mu$ which is impossible by the Proposition~\ref{neighdiamProp}, if $\lambda < 1/(328 C_1 C_{3,1})$. Thus, $y\cont{1} y + e_3$, and in the same fashion $y \cont{3} y + e_2$. Furthermore, apply the Proposition~\ref{tps} to $(y;y+e_2,y+e_1,y+e_3; y)$ with constant $6\rho(y)/\mu$ to get $d(y + e_2, y + e_1 + e_2) \leq 96\lambda \rho(y)$. 
\item[15] As in 11, we obtain $y \cont{1} y + e_1$ and $y \cont{3} y + e_3$. Apply the Proposition~\ref{tps} to $(y;y+e_3,y+e_2,y+e_1;y)$ with constant $6\rho(y)/\mu$ to get $d(y + e_2, y + 2e_2) \leq 96\lambda \rho(y)$.
\item[23] As in 11, we obtain $y \cont{3} y + e_1$ and $y \cont{1} y + e_3$. Apply the Proposition~\ref{tps} to $(y;y+e_1,y+e_2,y+e_3;y)$ with constant $6\rho(y)/\mu$ to get $d(y + e_2, y + 2e_2) \leq 96\lambda \rho(y)$.
\end{itemize}

\end{document}